\documentclass{siamart190516}
\usepackage{amsmath}
\usepackage{amssymb} 
\usepackage{graphicx}
\usepackage{amsfonts}
\usepackage{amssymb}
\usepackage{float}
\usepackage{color}
\usepackage{algorithm}
\usepackage{algorithmic}
\usepackage{hyperref}
\usepackage{savesym}
\usepackage{booktabs}
\usepackage{multirow}
\usepackage[normalem]{ulem}

\usepackage{multirow}
\usepackage{framed} 
\usepackage{xcolor} 
\usepackage{mathrsfs} 
\usepackage{euscript}

\newtheorem{remark}[theorem]{Remark}

\def\e{\mathrm{e}}
\def\irm{\mathrm{i}}

\def\d{\mathrm{d}} 
\def\ddt{\frac{\d}{\d t}}
\def\Df{\textsc{D}}
\def\Mx{\mathbf{M}_{X}}

\def\Mxa{\mathbf{M}_{X}^{-1}}
\def\gM{\mathrm{g}_{\mathbf{M}_{X}}}

\def\sym{\mathrm{sym}}
\def\skew{\mathrm{skew}}

\def\tr{\mathrm{tr}}
\def\diag{\mathrm{diag}}
\def\cay{\mathrm{cay}}

\def\gradm{\mathrm{grad}_{\mathbf{M}_{X}}}

\def\St{\mathrm{St}}
\def\iSt{\mathrm{iSt}}
\def\iStkn{\mathrm{iSt}(k,n)}

\def\Cbn{{\mathbb{C}^n}}
\def\Cb{{\mathbb{C}}}
\def\Rbkk{{\mathbb{R}^{k\times k}}}
\def\Rbnk{{\mathbb{R}^{n\times k}}}
\def\Rbnn{{\mathbb{R}^{n\times n}}}
\def\Rbln{{\mathbb{R}^{l\times n}}}

\def\R{\mathbb{R}}
\def\RX{\mathcal{R}}

\def\RZ{\mathcal{R}_X}
\def\RZi{\mathcal{R}_{X_j}}

\def\PX{\mathcal{P}_{X}}

\def\SPD{\mathrm{SPD}}
\def\O{\mathrm{O}}
\def\diag{\mathrm{diag}}

\def\calF{\mathcal{F}}

\newcommand{\skewset}{\mathrm{Skew}}
\newcommand{\symset}{\mathrm{Sym}}

\usepackage[utf8]{inputenc}

\usepackage{mathtools}
\hypersetup{colorlinks=true,linkcolor=blue,citecolor=blue}

\newcommand{\abs}[1]{\left|#1\right|}

\newcommand\restr[2]{\ensuremath{\left.#1\right|_{#2}}}
\def\restrict#1{\raise-.5ex\hbox{\ensuremath|}_{#1}}

\title{A Riemannian gradient descent  method\\for optimization on the indefinite Stiefel manifold 
}

\author{Dinh Van Tiep\thanks{
		Faculty of Fundamental and Applied Sciences, 
		Thai Nguyen University of Technology, 24131 Thai Nguyen, Vietnam (tiep.dv@tnut.edu.vn).}
	\and 
	Nguyen Thanh Son\thanks{Department of Mathematics and Informatics, Thai Nguyen University of Sciences, 24118 Thai Nguyen, Vietnam (ntson@tnus.edu.vn).}
}

\begin{document}
	\maketitle
	
	\begin{abstract} 	
		We consider  
		the optimization problem with a generally quadratic matrix constraint of the form $X^TAX = J$, where $A$ is a given nonsingular, symmetric $n\times n$ matrix and $J$ is a given $k\times k$ symmetric matrix, with $k\leq n$, satisfying $J^2 = I_k$. Since the feasible set constitutes a differentiable manifold, called the indefinite Stiefel manifold, we approach this problem within the framework of Riemannian optimization. Namely, we first equip the manifold with a Riemannian metric and construct the associated geometric structure,  then propose a retraction based on the Cayley transform, and finally suggest a Riemannian gradient descent method 
		using the attained materials, whose global convergence is guaranteed. Our results not only cover the known cases, the orthogonal and generalized Stiefel manifolds, 
		but also provide a Riemannian optimization solution for other constrained problems which has not been investigated. As applications, we consider, via trace minimization, several eigenvalue problems of symmetric positive definite matrix pencils, including the linear response eigenvalue problem, 
		and a matrix least square problem, a general framework for the Procrustes problem and constrained matrix equations. The 
		presented numerical results justify the theoretical findings. 
	\end{abstract}
	
	
	
	\pagestyle{myheadings}
	\thispagestyle{plain}
	
	\markboth{DINH VAN TIEP AND NGUYEN THANH SON}{RGD METHOD FOR OPTIMIZATION ON THE INDEFINITE STIEFEL MANIFOLD}
	

	
	\section{Introduction}\label{Sec:Intro}
	
	
	Optimization of smooth functions of variable $X\in\Rbnk$ with $k\leq n$ under the condition 
	\begin{equation}\label{eq:gen_constraint}
		X^TAX = B,
	\end{equation}
	in which $A\in\Rbnn, B \in \Rbkk$, possibly with some conditions imposed, 
	is popular thanks to its wide range of applications. In many circumstances 
	when the feasible set, $$\calF := \{X\in\Rbnk\ :\ X^TAX = B\},$$ constitutes a differentiable manifold, various algorithms were developed by extending popular methods for unconstrained optimization in the Euclidean space to the Riemannian manifold once necessary geometric tools such as the tangent space, metric,  projection, and retraction are efficiently realized. Such an approach is generally referred to as Riemannian optimization.
	
	The most well known situation is when $A = I_n$, the $n\times n$ identity matrix, and $B=I_k$, where $\calF$ is named as the \emph{(orthogonal) Stiefel manifold}. Numerous practical issues can be (re)formulated as an optimization problem on the Stiefel manifold, i.e., with the orthogonality constraint. 
	For instance, in dimensionality reduction of data, principal component analysis (PCA) 
	requires solving
	\begin{equation}\label{eq:PCAProb}
		\max_{X^T X = I_k}\tr\left(X^T\bar{S}^T\bar{S}X\right),
	\end{equation}
	where $\bar{S}$ is the centered data set, i.e., each of its columns is of the form $\bar{s}_j = s_j - \mu \in \R^n$ with $\mu = \frac{1}{m}\sum_{j=1}^ms_j$. Ky-Fan's theorem is a similar situation in numerical linear algebra which states that given a symmetric matrix $M$ whose eigenvalues are denoted by $\lambda_1 \leq \ldots \leq \lambda_n$, 
	we have
	\begin{equation}\label{eq:KyFan}
		\min_{X^TX = I_k}\tr\left(X^TMX\right) = \sum_{j=1}^k\lambda_j, \mbox{ for each } k = 1,\ldots,n.
	\end{equation}
	Riemannian optimization is among the methods of choice for solving \eqref{eq:PCAProb} and \eqref{eq:KyFan}. 
	The reader is referred to the classic work \cite{EdelAS98} and excellent monographs \cite{Udri1994,AbsiMS08,Sato2021,Boumal23} for details. 
	
	A slightly different context is when $A\in \SPD(n),$ the set of $n\times n$ symmetric positive-definite (spd) matrices, and $B = I_k$ occurring in the study of locality preserving projections where one wants to find a projection matrix that preserves some affinity graph from a given data set; it leads to the optimization problem
	\begin{equation}\label{eq:LPP}
		\min_{X^T(\hat{S}^T\hat{S})X = I_k}
		\tr\left(X^T\hat{S}^T(I-\hat{W})\hat{S}X\right),
	\end{equation}
	where $\hat{S}$ is generally of full rank and $\hat{W}$ is a modified weighted matrix  \cite{HeN03}. A similar problem arises in canonical correlation analysis  whose feasible set is the Cartesian product of two sets 
	as that of \eqref{eq:LPP} 
	\cite{Hote92,HardSST04}. 
	Meanwhile, an extension of Ky-Fan's theorem \cite{SameW82} for 
	the symmetric matrix pencil $(M,A)$, where $A$ is spd, also leads to 
	\begin{equation}\label{eq:KyFan2}
		\min_{X^{T}AX = I_k}\tr\left(X^{T}MX\right) = \sum_{j=1}^k\lambda_j,
		\mbox{ for each }k = 1,\ldots,n,
	\end{equation}
in which $\lambda_j, j = 1,\ldots,n,$ are the (generalized) eigenvalues of the matrix pencil $(M,A)$. Solutions to the optimization problems with constraints akin to \eqref{eq:LPP} and \eqref{eq:KyFan2} motivate the investigation of Riemannian optimization methods on the \emph{generalized Stiefel manifold}, denoted by $\St_A(k,n)$, \cite{YgerBGR12,manopt,MishS2016,SatoA19,ShusA2023,WangDPY2024}.
	These results provide us with tools to work directly on $\St_A(k,n)$ without having to turn back to $\St(k,n)$ via a Cholesky factorization of $A$ which should be avoided, especially for large and/or sparse $A$. More problems on minimization of the trace can be found in \cite{KokiS11}.
	
	Other authors 
	\cite{GSAS21,GSAS21a,BendZ21,OviH23,YamaS2023,GSS24,JenZ24,GaoSS2024} 
	concentrated on developing optimization methods on the \emph{symplectic Stiefel manifold}, which corresponds to the case where both $A$ and $B$ are the Poisson matrices, i.e., of the form $\left[\begin{smallmatrix}
		0&I_p\\-I_p&0
	\end{smallmatrix}\right]$. This problem arises in model reduction of Hamiltonian systems \cite{PengM16,BendZ22,BuchGH22,GSS24}, 
	computing the symplectic eigenvalues of (semi)spd matrices \cite{BhatJ15,SonAGS21,SonS22}. 
	
	In \cite{KovaV95}, Ky-Fan's theorem was further investigated for $n\times n$ symmetric positive-definite matrix pencils $M-\lambda A$, i.e., the ones such that 
	the matrix $M-\lambda_0 A$ 
	is spd for some 
	 $\lambda_0$. 
	In this extension,  $A$ is nonsingular and may be indefinite, i.e., $A$ has $p$ positive and $m$ negative eigenvalues, $p>0, m > 0, p+m = n$. Namely, let $J = \diag(I_{k_p},-I_{k_m}),$ the block diagonal matrix, 
with $0\leq k_p \leq p, 0\leq k_m \leq m, k_p + k_m = k$, and denote the eigenvalues of  $(M,A)$ by 
	$$
	\lambda^-_m\leq \ldots\leq \lambda^-_1 
	<\lambda^+_1 \leq \ldots \leq \lambda^+_p,
	$$
	where the minus/plus sign means the negative/positive value. Then we have
	\begin{equation}\label{eq:KovacStrikoVeselic}
		\min_{X^TAX = J}\tr\left(X^TMX\right) = \sum_{j=1}^{k_p}\lambda^+_j - \sum_{j=1}^{k_m}\lambda^-_j.
	\end{equation}
	Once a minimizer $X_*$ is obtained, one can determine the eigenvalues appearing on the right hand side of \eqref{eq:KovacStrikoVeselic} and their associated normalized eigenvectors, i.e., an $n\times k$ matrix $\hat{X}$ satisfying the constraint $\hat{X}^TA\hat{X}=J$ and its $j$-th column $\hat{x}_j$ verifies $M\hat{x}_j = \lambda_jA\hat{x}_j$ and $\lambda_j = \lambda_j^+$ for $j=1,\ldots,k_p$ and $\lambda_{j+k_p} = \lambda_j^-$ for $j=1,\ldots,k_m$, by solving two standard eigenvalue problems for the matrices of sizes $k_p\times k_p$ and $k_m\times k_m$ extracted from $X_*^TMX_*$, see \cite[Thm. 3.5]{KovaV95} and section~\ref{sec:NumerExam} for details. The need for computing several smallest positive eigenvalues of symmetric/Hermitian matrices/matrix pencils arises in practice, for instance, in the design of electromagnetic accelerators, see \cite{Geus2002} and references therein. In the case that $A=J=\diag(\pm 1,\ldots,\pm 1),$ the signature matrix, the problem of determining the eigenvalues and eigenvectors mentioned above is referred to as the hyperbolic eigenvalue problem \cite{SlapT2000}. A solution to this problem is a main step for hyperbolic PCA which is arguably more informative than the conventional PCA \cite{TabaKWM2024} and the computation of hyperbolic singular value decomposition \cite{OnnSB1989}. 
	
	Another issue is the linear response eigenvalue problem (LREVP) where 
	one has to compute a few smallest eigenvalues of the matrix $\left[ \begin{smallmatrix}
		0&K\\M&0
	\end{smallmatrix}\right]$ 
	with symmetric  positive-semidefinite matrices $K, M$ in which at least one of them is spd.
		This problem 
		originates 
		from the studies of density functions, random phase approximation, see \cite{BaiL12} and references therein. Various numerical methods have been proposed, see \cite{BaiL2017} for a recent survey, among which the ones based on trace minimization \cite{BaiL13,BaiL14,ChenSL2023} are of our interest. It is direct to see  that the mentioned LREVP is equivalent to the eigenvalue problem for the matrix pencil 
		\begin{equation}\label{eq:HGpencil}
			H - \lambda G := \begin{bmatrix}
				K&0\\0&M
			\end{bmatrix} - \lambda \begin{bmatrix}
				0&I\\I&0
			\end{bmatrix}.
		\end{equation} 
		When both $K$ and $M$ are spd, \eqref{eq:HGpencil} is equivalent to the product eigenvalue problem where one has to find the $k$ smallest eigenvalues of a product matrix $KM$  
		 \cite{KresPS2014,Pand2019}. Indeed, 
		it can be verified that 
		$\lambda^2, \lambda\in \R,$ 
		is an eigenvalue of $KM$ 
		if and only if $\pm\lambda$ are the eigenvalues of the matrix pencil \eqref{eq:HGpencil}.
		The statement \eqref{eq:KovacStrikoVeselic} suggests that the LREVP 
		can be solved via finding a solution  $X\in \R^{2n\times k}$ to
		\begin{equation}\label{eq:IndefTraceMin}
			\min_{X^TGX = J }\tr\left(X^THX\right)
		\end{equation}
		with $H$ and $G$ defined as in \eqref{eq:HGpencil} and $J = I_k$.

		Obviously, the existing Riemannian optimization methods on the generalized Stiefel manifold and the symplectic Stiefel manifold do not help 
		solve the optimization problems mentioned in \eqref{eq:KovacStrikoVeselic} and \eqref{eq:IndefTraceMin}. The main reason 
		lies in the difference in 
		the feasible sets. This fact inspires us 
		to consider
		the optimization problem 
		\begin{equation}\label{eq:opt_prob}
			\min_{X\in \R^{n\times k}} f(X)\ \mbox{ s.t. }\ X^{T}AX = J,
		\end{equation}
		where $f$ is continuously differentiable, $A\in \R^{n\times n}$ is symmetric, nonsingular and 
		\mbox{$J\in\symset(k)$} satisfying $J^2 = I_k$. As will be shown in section~\ref{sec:Geometry}, the feasible set of this problem 
		\begin{equation}\label{eq:indefStiefel}
			\iSt_{A,J}(k,n) = \{X\in \R^{n\times k}\ :\ X^{T}AX =  J\}
		\end{equation}
		constitutes a differentiable manifold. We will refer to it as the \emph{indefinite Stiefel manifold} to emphasize the possible indefiniteness of $A$ and $J$\footnote{There is a mixed use of terms for \textit{``Stiefel-type"} manifolds. The set   $\iSt_{A,I_k}$, which is a special case of $\iSt_{A,J}$ in 
		\eqref{eq:indefStiefel}, was named the \emph{indefinite Stiefel manifold} in \cite{Kobayashi1992}, \emph{pseudo-Stiefel manifold} in \cite[Remark 1.3.2]{Chikuse2003} and \emph{generalized Stiefel manifold} in \cite{Li1993,SenadoMendoza2022}. At the same time, as mentioned previously, the \emph{generalized Stiefel manifold} means $\iSt_{A,I_k}$  with $A$ is spd in the optimization community and in this work. The set $\{X\in \Rbnk :\ X^TX = C\}$, where $C$ is spd, was called 
		the \emph{Stiefel C-manifold} in \cite{Downs1972}.}. 
	If the two matrices $A,J$ or the size $n\times k$ of the ambient space are either clear from the context and/or inessential, we also write $\iSt(k,n)$ or $\iSt_{A,J}$. 
	This manifold covers several sets considered in the literature. If $A$ is additionally positive-definite and $J = I_k$, the constraint set in our problem is a generalized Stiefel manifold which naturally reduces to the Stiefel manifold if $A=I_n$. For $A = \diag(-1,1,\ldots,1)$, if $J = I_k$, we obtain the hyperbolic manifold mentioned in 
		\cite{XiaoLT2024,HanJM2024}; 
		if $J = -1$, the Lorentz model or the hyperboloid model of hyperbolic space, see, e.g., \cite{NickK2018,WilsL2018}, is a half of $\iSt_{A,J}$. Moreover, if $A = J = \diag(\pm 1,\ldots,\pm 1)$, we are having the $J$-orthogonal group \cite{High03,HeYWX2024}.
		
		Despite the similarity, there are essential differences between the indefinite Stiefel manifold and other Stiefel-type manifolds. 
		The generalized Stiefel manifold, including the Stiefel manifold as a special case, is compact while the indefinite Stiefel manifold is generally not so. This can be easily seen by considering the hyperboloid model $x_1^2 - x_2^2 = 1$ in $\R^{2\times 1}$. Moreover, $\iSt_{A,J}$ in this paper is defined generally, both $J$ and $A$ are symmetric (with appropriate restrictions), and therefore completely different from the symplectic Stiefel manifold defined in \cite{GSAS21} which is attached to two fixed (skew-symmetric) Poisson matrices $J_{2n}$ and $J_{2k}$. 
		

		\paragraph{Contribution and outline} To the best of our knowledge, our present work is the first attempt to approach the optimization problem with possibly indefinite quadratic constraint $X^TAX = J$ in the Riemannian framework. Namely, in section~\ref{sec:Geometry} we show  that $\iSt_{A,J}$ is an embedded submanifold of $\Rbnk$ 
		and describe its tangent vectors. We then equip the manifold with a family of tractable metrics, construct the associated orthogonal projections onto the tangent and normal spaces, and compute the Riemannian gradient of a general cost function in section~\ref{sec:Metric}. A retraction based on the Cayley transformation is studied in section~\ref{sec:Retraction} and a gradient descent algorithm is proposed in section~\ref{sec:Algorithm}. The obtained results are then verified by various numerical examples in section~\ref{sec:NumerExam}. 
		Since the results in this work are not necessarily restricted to the indefinite case, they also hold for the generalized Stiefel manifold; with suitable restrictions, 
		some of them reduce to the results derived for generalized Stiefel manifold, see remarks~\ref{rem:Compare_TangentVect}, \ref{rem:Compare_Metric}, and \ref{rem:Compare_Cayley}. Moreover, our result enables a Riemannian approach for optimization with $J$-orthogonality constraint. 
		
		
		\paragraph{Notation}\label{ssec:Notation}
		First, $\dim(\cdot)$ will denote the dimension of a linear space, $\diag(\cdot,\ldots,\cdot)$ is the (block) diagonal matrix whose  diagonal (block) entries are the listed objects. For any matrix $\Omega$, the notations $\Omega^T, \tr\,\Omega,\, \|\Omega\|_F,$ and $\|\Omega\|$ are respectively its transpose, trace, Frobenius norm, and spectral norm, if not specified otherwise. If $\Omega$ is a square matrix, we denote by $\skew(\Omega) \coloneqq \frac{1}{2}(\Omega-\Omega^{T})$ and $\sym(\Omega) \coloneqq \frac{1}{2}(\Omega + \Omega^{T})$ 
		its skew-symmetric and symmetric parts, respectively. Meanwhile, $\skewset(p)$,  $\symset(p)$, and $\O(p)$ 
			stand for the set of skew-symmetric, the set of symmetric, and the set of orthogonal matrices of size $p\times p,$ 
			respectively. The Fréchet derivative of a map $F:\mathcal{V}_1\longrightarrow\mathcal{V}_2$ at $X\in\mathcal{V}_1$ between two normed vector spaces $\mathcal{V}_1$ and $\mathcal{V}_2$ is defined to be a linear operator $\Df F(X):\mathcal{V}_1\longrightarrow\mathcal{V}_2$ such that 
		$F(X+Z) = F(X)+\Df F(X)[Z]+o(\| Z\|),$ for any $Z\in\mathcal{V}_1,$ where $o(\cdot)$ is the little-$o$ notation. Finally, for a smooth mapping $\gamma$ of one variable $t$, $\ddt(\gamma)$ or $\dot{\gamma}(t)$ 
			denotes the derivative of $\gamma$ with respect to $t.$
		\section{The indefinite Stiefel manifolds}\label{sec:Geometry} Without further restriction on $A$ and $J$, the set $\iStkn$ can be empty. For instance, there is no real number pair $(x_1,x_2)$ such that $x_1^2 + x_2^2 = -1$. The following statement helps avoid working with such a set.
		\begin{lemma}\label{lem:nonempty}
		The set $\iSt_{A,J}(k,n)$ defined in \eqref{eq:indefStiefel} is nonempty if and only if 
		\begin{equation}\label{eq:lem_nonempty}
		\irm_+(J) \leq \irm_+(A) \mbox{ and }  \irm_-(J) \leq \irm_-(A),
	\end{equation} 
where $\irm_+(M)$ and $\irm_-(M)$ denote the number of positive and negative eigenvalues of a real symmetric $M$, respectively.
	\end{lemma}
\begin{proof}
The necessity for the nonemptiness of $\iSt_{A,J}(k,n)$ follows directly from \cite[Lem. 1.6]{Tian2010}. Moreover, by assumption, there exists $n\times k$ real matrix $V$ such that $V^TAV = \diag(I_{\irm_+(J)},-I_{\irm_-(J)})$ and $k\times k$ orthogonal matrix $U$ such that $U^TJU = \diag(I_{\irm_+(J)},-I_{\irm_-(J)})$. As a consequence, $\iSt_{A,J}(k,n)$ has the matrix $VU^T$ as one of its element and therefore verifies the sufficiency.
\end{proof}

From now on, it is essential to assume that the condition \eqref{eq:lem_nonempty} is always fulfilled. The following result gives the foundation for $\iStkn.$ 
		\begin{proposition}\label{theo:PropertiesH}
			The set $\iStkn$ defined in \eqref{eq:indefStiefel} is a closed, embedded submanifold of $\Rbnk$ with the dimension 
			$\dim(\iStkn)=nk-\frac{1}{2}k(k+1)$.
		\end{proposition}
		\begin{proof}
			Define the map
			\begin{eqnarray*} 
				F: \Rbnk  & \longrightarrow &\symset(k) \\
				X  &\longmapsto &  X^{T}AX-J
			\end{eqnarray*}
			Since F is continuous and $F^{-1}(0) = \iStkn,$ it follows that $\iStkn$ is closed. For any fixed  $X\in\iStkn$, we are going to verify that 
			$\Df F(X)$ is a surjection. Indeed, since
			\begin{equation}\label{eq:eqDF}
				\Df F(X)[Z]=X^{T}AZ+Z^{T}AX,\ \forall\ Z\in\Rbnk,
			\end{equation}
			and $J^{T}=J,$ for each $Y\in\symset(k),$ with $Z \coloneqq\frac{1}{2}XJY$, we have   
			\begin{center}
				$\Df F(X)[Z] = \dfrac{1}{2}X^{T}AXJY + \dfrac{1}{2}Y^{T}J^{T}X^{T}AX = Y.$
			\end{center}
			Next, by the Regular Level Set Theorem \cite[Corollary~5.14]{Lee12}, 
			it follows that $\iStkn$ is a closed embedded submanifold of $\Rbnk$ and 
			\begin{align*}
					\dim(\iStkn) = 
					\dim(\Rbnk) - \dim \symset(k)  
					= nk-\dfrac{1}{2}k(k+1),
				\end{align*}
				which completes the proof.
		\end{proof}
		
		A proof for a special case when both $A$ and $J$ are of the form $\diag(\pm 1,\ldots,\pm 1)$ was given in 
		\cite[Thm. 3.5]{KovaV95}. It can be used to obtain the statement made in Proposition~\ref{theo:PropertiesH} via a suitable diffeomorphism but the given direct proof better prepares the reader for later development. 
		
		Let us denote by $\langle\cdot,\cdot\rangle_{\R^{p\times q}}$  the standard Euclidean inner product on $\R^{p\times q}$. That is, $\langle U,V\rangle_{\R^{p\times q}} =\tr(U^TV)$ for $U,V\in\R^{p\times q}$. Given $X\in\iSt_{A,J},$ we are going to formulate the adjoint operator $\Df F^{*}(X)$ 
			of $\Df F(X)$ with respect to this inner product, 
			which is the linear operator
			$$\Df F^{*}(X) :\symset(k) \longrightarrow \Rbnk$$
			satisfying the condition
			\begin{equation}\label{eq:DF_adjoint}
				\langle\Df F(X)[Z],\Omega\rangle_{\Rbkk} = \langle Z,\Df F^{*}(X)[\Omega]\rangle_{\Rbnk},\ \forall\ \Omega\in \symset(k),\   Z\in\Rbnk,
			\end{equation}
			and show that $\Df F^{*}(X)$ is injective.
		Indeed, since
		\begin{equation*}
				\langle\Df F(X)[Z],\Omega\rangle_{\Rbkk} = \langle  X^{T}AZ+Z^{T}AX,\Omega\rangle_{\Rbkk} 
				= 2\tr(Z^{T}AX\Omega)\\
				= \langle Z,2AX\Omega\rangle_{\Rbnk},
			\end{equation*}
			in view of 
			\eqref{eq:DF_adjoint}, we deduce  that 
			\begin{equation}\label{eq:DF_adjoint2}
				\Df F^{*}(X)[\Omega] = 2AX\Omega,\mbox{ for } \Omega\in\symset(k).
		\end{equation}
		Combining with the fact that 
		$X^{T}AX =J$ is invertible, 
		it follows that 
		$\Df F^{*}(X)$ is injective. 
		
		For 
			$X\in\iSt_{A,J}$, let us denote by $X_{\perp}\in\mathbb{R}^{n\times(n-k)}$ a full rank matrix such that $X^{T}X_{\perp}=0\in\mathbb{R}^{k\times(n-k)}.$ Note that $X_{\perp}$ is not unique. The following technical lemma,  similar to  \cite[Lem. 3.2]{GSAS21}, will be useful later. 
			
			\begin{lemma}\label{prop:eleproperty} 
				For $X\in \iStkn$, let us denote the matrix
				\begin{equation*}
					E\coloneqq [X \;\; A^{-1}X_{\perp}]\in \Rbnn.
				\end{equation*} Then
				\begin{itemize}
					\item[i)] $E$ is nonsingular; 
					\item[ii)] $E^{T}AE = 
					\begin{bmatrix}
						J & 0 \\ 
						0 & X_{\perp}^{T}A^{-1}X_{\perp} 
					\end{bmatrix}$ and $X_{\perp}^{T}A^{-1}X_{\perp}$ is nonsingular;
					\item[iii)] $E^{-1}=
					\begin{bmatrix}
						JX^{T}A \\ 
						(X_{\perp}^{T}A^{-1}X_{\perp})^{-1}X_{\perp}^{T} 
					\end{bmatrix};$
					\item[iv)] For any $Z\in\Rbnk,$ there are $W\in\Rbkk$ and $K\in\mathbb{R}^{(n-k)\times k} $ such that
					\begin{equation}
						Z = XW + A^{-1}X_{\perp}K.\label{eq:eq21}
					\end{equation}
					Moreover, $W = JX^{T}AZ, K = (X_{\perp}^{T}A^{-1}X_{\perp})^{-1}X_{\perp}^{T}Z,$ and $A^{-1}X_{\perp}K$ is independent of $X_{\perp}.$
				\end{itemize} 	
			\end{lemma}
			\begin{proof}
				\begin{itemize}\item[i)]
					To proceed, we will verify that the homogeneous linear system $E \left[\begin{smallmatrix}
							y_1 \\ 
							y_2
						\end{smallmatrix}\right]=0$ 
						with $y_1\in\mathbb{R}^{k}, y_2\in\mathbb{R}^{n-k}$ has only the trivial solution. 
					Multiply both sides of the equation by $X^{T}A$ from the left to obtain 
						$Jy_1 = 0$. This results in 
					$A^{-1}X_{\perp}y_2 = 0$ which immediately implies that $y_2=0$ 
					because $A^{-1}X_{\perp}\in\mathbb{R}^{n\times (n-k)}$ is of full rank.
					\item[ii)] The identity follows from a direct calculation. 
					Moreover, since the matrices $E$ and $A$ are invertible, so are $E^{T}AE$ and its diagonal block $X_{\perp}^{T}A^{-1}X_{\perp}.$
					\item[iii)] The equality in ii) yields that 
						$	E^{-1} = \left[\begin{smallmatrix}
							J & 0\\ 
							0 & (X_{\perp}^{T}A^{-1}X_{\perp})^{-1}
						\end{smallmatrix}\right] E^TA
						$
						which leads to the expected equality. 
					\item[iv)] For any $Z\in\Rbnk,$ let $W\in\Rbkk$ and $K\in\mathbb{R}^{(n-k)\times k}$ be the matrices obtained from the identity $E^{-1}Z =  \left[\begin{smallmatrix}
						W\\ 
						K
					\end{smallmatrix}\right].$ From iii), we get $W = JX^{T}AZ$ and $ K = (X_{\perp}^{T}A^{-1}X_{\perp})^{-1}X_{\perp}^{T}Z.$ 
					So,
					$Z = E  \left[\begin{smallmatrix}
						W\\ 
						K
					\end{smallmatrix}\right] = [X\; A^{-1}X_{\perp}] \left[\begin{smallmatrix}
						W\\ 
						K
					\end{smallmatrix}\right] = XW +  A^{-1}X_{\perp}K.$ 
					The independence of $X_{\perp}$  
						is derived from the fact 
						that $A^{-1}X_{\perp}K = Z - XW$ with $W = JX^TAZ$.
				\end{itemize}
			\end{proof}
			
			Next, we determine the tangent space at a point on the manifold $\iStkn.$ This is crucial for optimization with manifold constraint since most iterative methods update the iterate 
			along a vector belonging to this space. Recall that 
			the tangent space at $X\in  \iStkn$ is the set 
				\begin{equation*}
					T_{X}\iStkn = \left\lbrace \dot{\varphi}(0) \;  |\;  \varphi: I\ni 0 \longrightarrow \iStkn \text{ is a smooth curve and }\varphi(0)=X \right\rbrace.
			\end{equation*} 
			The following proposition characterizes the elements of this tangent space.
			
			\begin{proposition}\label{prop:TxH}
				For a given $X\in \iStkn$, the tangent space $T_{X}\iStkn$ {can be expressed as follows:} 
				\begin{subequations}
					\begin{align}
						T_{X}\iStkn & = \left\{Z\in\Rbnk: Z^{T}AX + X^{T}AZ = 0\right\}\label{eq:eq31a} \\
						& = \left\{XW + A^{-1}X_{\perp}K: JW\in\skewset(k),    K\in\mathbb{R}^{(n-k)\times k} \right\} \label{eq:eq31b}\\
						& = \left\{   SAX: S\in\skewset(n)  \right\}. \label{eq:eq31c}
					\end{align}
				\end{subequations}
			\end{proposition}
			\begin{proof} 
				Let $F$ be as in the proof of Proposition~\ref{theo:PropertiesH}. In view of \cite[Proposition 5.38]{Lee12}, $T_{X}\iStkn = \mathrm{ker}( \Df F(X))$, which together with \eqref{eq:eqDF}, yields \eqref{eq:eq31a}. Note that using the definition, \eqref{eq:eq31a} can also be derived by differentiating the relation $X(t)^TAX(t) = J$ at $0$. 
				
				To verify \eqref{eq:eq31b}, we use \eqref{eq:eq21} in Lemma \ref{prop:eleproperty}. For any $Z\in\Rbnk,$ there are $W\in\Rbkk$ and $K\in\mathbb{R}^{(n-k)\times k}$ such that 
				$Z = XW + A^{-1}X_{\perp}K.$ Therefore, 
				\begin{align*}
					Z^{T}AX + X^{T}AZ & = (W^{T}X^{T} + K^{T}X_{\perp}^{T}A^{-1})AX + X^{T}A(XW + A^{-1}X_{\perp}K)\\
					& = W^{T}J + JW = (JW)^{T}+JW.
				\end{align*}
				By \eqref{eq:eq31a}, $Z\in T_{X}\iStkn$ if and only if $JW=-(JW)^{T}\in\skewset(k).$ 
				
				Finally, we see that $\left\{   SAX: S\in\skewset(n)  \right\}$ is a linear subspace of $ T_{X}\iStkn$ since 
				$(SAX)^{T}AX+X^{T}A(SAX)=X^{T}AS^{T}AX+X^{T}ASAX=0.$ 
				Hence, to verify \eqref{eq:eq31c}, 
				it is sufficient to prove that
				$\dim \left\{   SAX: S\in\skewset(n)  \right\} = \dim \iStkn$ which is also the dimension of $T_{X}\iStkn$.  
				Indeed, let $S\in \skewset(n)$ and $P = AE = [AX\ \ X_{\perp}],$ where $E$ is as in Lemma \ref{prop:eleproperty}. Then $P$ is nonsingular and $AX = P\left[\begin{smallmatrix}
					I_k\\ 
					0
				\end{smallmatrix}\right].$ So $P^{T}SAX = P^{T}SP\left[\begin{smallmatrix}
					I_k\\ 
					0
				\end{smallmatrix}\right] =: B\left[\begin{smallmatrix}
					I_k\\ 
					0
				\end{smallmatrix}\right],$ where $B \coloneqq P^{T}SP\in\skewset(n).$ 
				
				Assume that $B =\left[\begin{smallmatrix}
					B_{11}&-B_{21}^{T}\\ 
					B_{21}&B_{22}
				\end{smallmatrix}\right],$ with $B_{11}\in\skewset(k), B_{22}\in\skewset(n-k)$ and $B_{21}\in\mathbb{R}^{(n-k)\times k}.$ Then, $P^{T}SAX = B\left[\begin{smallmatrix}
					I_k\\ 
					0
				\end{smallmatrix}\right] = \left[\begin{smallmatrix}
					B_{11}\\ 
					B_{21}
				\end{smallmatrix}\right].$ Therefore,
				\begin{align*}
					\dim\left\{   SAX: S\in\skewset(n)  \right\} &= \dim\left\{   P^{T}SAX: S\in\skewset(n)  \right\} \\
					&= \dim\left\{\begin{bmatrix}
						B_{11}\\ 
						B_{21}
					\end{bmatrix}: B_{11}\in\skewset(k), B_{21}\in\mathbb{R}^{(n-k)\times k}\right\} \\
					&= \dim\skewset(k) + \dim\mathbb{R}^{(n-k)\times k} \\
					&= \dfrac{1}{2}k(k-1) + (n-k)k = nk - \dfrac{1}{2}k(k+1) \\
					& = \dim\iStkn.
				\end{align*}
				Here, the first identity follows from the nonsingularity of $P$ and $P^{T}$ whereas the last one has just been obtained in Proposition~\ref{theo:PropertiesH}. 
			\end{proof}
				\begin{remark}\label{rem:Compare_TangentVect}
					The formula \eqref{eq:eq31a} was derived for several manifolds as a quadratic constraint, see \cite[(3.24)]{AbsiMS08}, \cite[sect. 3.1]{YgerBGR12}, and \cite[(3.8a)]{GSAS21} while \eqref{eq:eq31b} is an extension of \cite[(2.5)]{EdelAS98}, \cite[sect. 3.1]{YgerBGR12}; cf., \cite[(3.8b)]{GSAS21}.
			\end{remark}
			\section{Metric and geometric objects 
				on $\iSt_{A,J}$}\label{sec:Metric}
				A metric $\mathrm{g}_X$ on the manifold $\iSt_{A,J}$ is an inner product defined on each tangent space $T_X\iSt_{A,J}$ which is required to depend smoothly on $X$. In this case, the pair $(\iSt_{A,J}, \mathrm{g}_X)$ becomes a Riemannian manifold and the metric is referred to as a Riemannian metric. 
		On a Riemannian manifold, a metric allows quantification such as length, distance, angle and, particularly, the definition of orthogonality which is essential in designing many optimization methods. 
	
	There are numerous ways to select
		a metric for a given smooth manifold. This freedom is exploited to get advantages in computation in the framework of preconditioning, see \cite{MishS2016,DongGGG2022,ShusA2023,GaoPY2024}, 
		to name a few. To this end, instead of a single one, in this section we will equip $\iSt_{A,J}$ with a family of tractable metrics, which was recently introduced in \cite{NguyD23}, used for the generalized Stiefel manifold \cite{ShusA2023} and the symplectic Stiefel manifold \cite{GaoSS2024}, and present the resulting geometric structure. 

Let $\Mx\in\SPD(n)$ depend 
smoothly on $X\in\iStkn$. We define an inner product $\gM$ on $T_{X}\iStkn$ by
\begin{equation}\label{eq:def_gMx}
	\gM(Z_1,Z_2)=\tr(Z_1^{T}\Mx Z_2), \text{ for all } Z_1, Z_2\in T_X\iStkn.
\end{equation}
It can be seen that $\gM$ is indeed a metric not only on $\iStkn$ 
but also on $\Rbnk$ by letting $Z_1, Z_2$ vary in $\Rbnk$. 
	As a result, we can define the \textit{normal space} $T_{X}\iStkn^{\perp}_{\gM}$ with respect to the inner product $\gM$ of $\Rbnk$ as
	\begin{equation*}
		T_{X}\iStkn^{\perp}_{\gM} \coloneqq \left\{ N\in\Rbnk: \gM(N,Z)=0,\text{ for all } Z\in T_{X}\iStkn\right\}.
	\end{equation*}
	The following characterization of $T_{X}\iStkn^{\perp}_{\gM}$ 
	helps 
	formulate the orthogonal projection onto $T_{X}\iStkn$. 
\begin{proposition}\label{prop:normalspace_gM}
	The normal space 
		to $\iStkn$ at $X\in \iStkn$ with respect to $\gM$ is given by	$T_{X}\iStkn^{\perp}_{\gM}=\left\lbrace \Mxa AXW \in\Rbnk,\; W\in\symset(k)\right\rbrace.$
\end{proposition}
\begin{proof}
	
	Consider $E\in\Rbnn$ defined as in Lemma~\ref{prop:eleproperty} and any $N\in\Rbnk$.
	In view of Lemma~\ref{prop:eleproperty}.(iv), 
	there exist unique matrices $\tilde{W}_N\in\Rbkk$ and $\tilde{K}_N\in\mathbb{R}^{(n-k)\times k}$ such that
	$A^{-1}\Mx N = E\left[\begin{smallmatrix}
		\tilde{W}_N\\
		\tilde{K}_N
	\end{smallmatrix}\right]$ which yields 
	$N = \Mxa AE\left[\begin{smallmatrix}
		\tilde{W}_N \\
		\tilde{K}_N
	\end{smallmatrix}\right].$ 
	Next, using the form \eqref{eq:eq31b} for $Z\in T_{X}\iStkn,$ we can write 
	$Z = XW_Z+A^{-1}X_{\perp}K_Z = E \left[\begin{smallmatrix}
		W_Z\\
		K_Z
	\end{smallmatrix}\right].$ 
	We obtain then
	\begin{align*}
		N^{T}\Mx Z 
		=& \left[ \tilde{W}_N^{T}\ \ \tilde{K}_N^{T}\right]E^{T}AE \left[\begin{matrix}
			W_Z\\
			K_Z
		\end{matrix}\right]\\
		=& \left[ \tilde{W}_N^{T}\ \ \tilde{K}_N^{T}\right] \begin{bmatrix}
			J & 0 \\ 
			0 & X_{\perp}^{T}A^{-1}X_{\perp} 
		\end{bmatrix} \left[\begin{matrix}
			W_Z\\
			K_Z
		\end{matrix}\right]\\
		=& \tilde{W}_N^{T}JW_Z + \tilde{K}_N^{T}(X_{\perp}^{T}A^{-1}X_{\perp})K_Z,
	\end{align*}
	where the second 
	identity follows from Lemma~\ref{prop:eleproperty}.(ii). Hence 
	$$\gM(N,Z) =\tr(N^{T}\Mx Z) = \tr\left( \tilde{W}_N^{T}JW_Z\right) +\tr\left(\tilde{K}_N^{T}(X_{\perp}^{T}A^{-1}X_{\perp})K_Z\right).$$
	When $Z$ varies all over $T_X\iStkn,$ 
	one can check that $JW_Z$ and $(X_{\perp}^{T}A^{-1}X_{\perp})K_Z$ will vary all over the sets 
	$\skewset(k)$ and $\R^{(n-k)\times k},$ respectively. 
	Therefore, $\gM(N,Z) = 0$ 
	for all $Z\in T_X\iStkn$ if and only if 
	$\tilde{W}_N\in\symset(k)$ and $\tilde{K}_N=0.$ That is, a normal vector must take the form $N =  \Mxa AX\tilde{W}_N$ with $\tilde{W}_N\in\symset(k).$ This completes the proof. 
		\end{proof}	
		
		Next, for each $X\in \iSt_{A,J}$, considering $\gM$ as an inner product on the ambient space $\Rbnk$, it holds that
		$$\Rbnk = T_X\iStkn\  
		{\bigoplus}\ T_X\iStkn_{\gM}^{\perp}.$$ 
		So, 
		there exists a unique pair $(Y_1,Y_2)\in T_X \iStkn\times T_X\iStkn_{\gM}^{\perp}$ for each $Y$ in $ \Rbnk$ 
		such that $Y = Y_1 + Y_2.$ This enables us to define two orthogonal projections 
		\begin{align*}
			\PX^{\Mx}: \Rbnk &\longrightarrow T_{X}\iStkn\\
			Y& \longmapsto \PX^{\Mx}(Y) = Y_1
		\end{align*}
		and 
		\begin{align*}
			\PX^{\perp_{\Mx}}: \Rbnk&\longrightarrow T_{X}\iStkn_{\Mx}^{\perp}\\
			Y&\longmapsto \PX^{\perp_{\Mx}}(Y) = Y_2.
		\end{align*} 
		The following proposition provides 
		formulae for the two projections above. 
		\begin{proposition}\label{prop:proj_Mx}
			The orthogonal projections of $Y\in\Rbnk$ with respect to metric $\gM$ onto $T_{X}\iStkn$ and $T_{X}\iStkn_{\Mx}^{\perp}$ are
				\begin{equation}\label{eq:orth_projs}
					\PX^{\Mx}(Y) = Y - \Mxa AXU_{X,Y}
					\mbox{ and } \PX^{\perp_{\Mx}}(Y) = \Mxa AXU_{X,Y},
				\end{equation}
				respectively, where $U_{X,Y}$ is the solution to the Lyapunov equation 
			\begin{equation}\label{eq:eqLya_gMx}
				(X^{T}A\Mxa A X)U+U(X^{T}A\Mxa A X) = 2\;\sym(X^{T}AY),
			\end{equation}
			with unknown $U\in\Rbkk.$
		\end{proposition}
		\begin{proof}
			Using \cite[Prop. 3.1]{NguyD23}, we have that
			\begin{equation*}
				\PX^{\Mx}(Y) = Y - \Mxa \Df F^*(X)\left( \Df F(X)\Mxa \Df F^*(X)\right)^{-1}\Df F(X)[Y]. 
			\end{equation*}
			For $U\in\symset(k),$ taking \eqref{eq:eqDF} and \eqref{eq:DF_adjoint2} into account, we obtain 
			\begin{equation*}
				\Df F(X)\Mxa \Df F^{*}(X)[U] = 2X^{T}A\Mxa AXU + 2U X^{T}A\Mxa AX.
			\end{equation*}
			Therefore, $\left( \Df F(X)\Mxa \Df F^*(X)\right)^{-1}\Df F(X)[Y]$ is 
				the solution to the Lyapunov equation
				\begin{equation}\label{eq:orthproj_proof}
					2X^{T}A\Mxa AXU + 2U X^{T}A\Mxa AX =  \Df F(X)[Y] = 2\;\sym(X^{T}AY).
				\end{equation}
				Its existence, uniqueness, and symmetry are ensured by the facts that $(X^{T}A\Mxa A X)$ is spd and the right hand side of  \eqref{eq:orthproj_proof} is symmetric, see  \cite[sect. 5]{Simo16} and references therein. Comparing \eqref{eq:orthproj_proof} to \eqref{eq:eqLya_gMx} 
				implies 
				that $2\left( \Df F(X)\Mxa \Df F^*(X)\right)^{-1}\Df F(X)[Y] =U_{X,Y}$ 
				is the solution to \eqref{eq:eqLya_gMx}. Then, by \eqref{eq:DF_adjoint2} again, we get that
				$$
				\PX^{\Mx}(Y) = Y - \Mxa \Df F^*(X)[\frac{1}{2}U_{X,Y}] = Y - \Mxa AXU_{X,Y}.
				$$
				The second formula in \eqref{eq:orth_projs} holds automatically by definition.
		\end{proof}
		
		Based on the presented derivation of the orthogonal projection onto the tangent space, we examine  the Riemannian gradient. Let $f:\iStkn\longrightarrow\R$ be a continuously differentiable function. The \textit{Riemannian gradient} of $f$ with respect to the tractable metric $\gM$ is the unique vector $\gradm f(X)$ of $T_X\iStkn$ satisfying 
		\begin{equation*}
			\gM\left( \gradm f(X),Z\right) = \Df\bar{f}(X)[Z]= \langle\nabla \bar{f}(X),Z\rangle_{\Rbnk},\ \forall\ Z\in T_X\iStkn,
		\end{equation*}
		where 
		$\bar{f}$ is any smooth extension of $f$ on a neighborhood of $X$ in $\Rbnk,$ and $\nabla \bar{f}(X)$ is the Euclidean gradient of $\bar{f}$ at $X.$ This definition  
			is also used for a function defined on the ambient space $\Rbnk$ when one wants to refer to its gradient associated with a metric other than the conventional Euclidean one. Next, we will obtain an explicit formula for the Riemannian gradient. 
		\begin{proposition} Let $f$ and $\bar{f}$ be functions as aforementioned. We have
			\begin{equation}\label{eq:gradMx_formula}
				\gradm f(X) = \Mxa\nabla\bar{f}(X) - \Mxa AX U_{\bar{f}},
			\end{equation}
			where $U_{\bar{f}}$ 
				is the unique solution to the Lyapunov equation with unknown $U$: 
			\begin{equation*}
				(X^{T}A\Mxa A X)U+U(X^{T}A\Mxa A X) = 2\sym\left[ X^{T}A\Mxa\nabla\bar{f}(X)\right].
			\end{equation*}
		\end{proposition}
		\begin{proof}
			In view of \eqref{eq:def_gMx}, the gradient of $\bar{f}$ at $X$ with respect to the metric $\gM$ (extended to $\Rbnk$) is determined by 
				\begin{equation*}
					\langle\Mx \gradm \bar{f}(X),Z\rangle_{\Rbnk} = \gM(\gradm \bar{f}(X),Z) =  
					\langle\nabla \bar{f}(X),Z\rangle_{\Rbnk}\ \forall\  Z\in\Rbnk.
				\end{equation*}
			Hence, $\gradm\bar{f}(X) = \Mxa \nabla \bar{f}(X).$ Taking  \cite[(3.27)]{AbsiMS08} into account, we obtain that 
			$\gradm f(X) = \PX^{\Mx}\left(\gradm \bar{f}(X)\right) = \PX^{\Mx}\left(\Mxa \nabla \bar{f}(X)\right)$. Finally, in view of \eqref{eq:orth_projs}, we obtain \eqref{eq:gradMx_formula}.
			\end{proof}
			
			In many practical cases, the matrix $\Mx$ representing the tractable metric is chosen to be constant. 
			We will refer to this case as the \emph{weighted Euclidean metric}. 
				When $\Mx$ is further set to be the identity matrix $I_n$, $\gM$ reduces to the popular Euclidean metric. 
				
				A specific choice for the metric that can help avoid having to solve the Lyapunov equation and leads to a different geometric structure and optimization tools will be addressed in an ongoing work.
				
		\begin{remark}\label{rem:Compare_Metric}
				It turns out that the formulations of the normal space \cite[(3.11)]{ShusA2023} and the two orthogonal projections \cite[Lem. 3.1]{ShusA2023} for the generalized Stiefel manifold hold in a broader context, the indefinite Stiefel manifold, see 
				Proposition~\ref{prop:normalspace_gM} and \eqref{eq:orth_projs}; cf. \cite[Prop. 1 and Prop. 2]{GSS24}.
		\end{remark}

\section{Cayley retraction} 
\label{sec:Retraction}
In this section, we present a tool 
which allows us to update a solution to 
the problem \eqref{eq:opt_prob}
along a given direction in the tangent space $T_{X}\iStkn.$ Namely, we define a retraction \cite{AdleDMMS02,AbsiMS08} on $\iStkn$ to be a smooth mapping $\RX$ from the tangent bundle $T\iStkn\coloneqq \bigcup_{X\in\iStkn}T_X\iStkn$ to $\iStkn$ satisfying the following conditions for all $X\in\iStkn$:
\begin{align*}
\tag{R.1} \; \;\; &\RZ(0) = X, \text{ where } 0 \text{ is the origin of } T_X\iStkn,\label{eq:conditionC1}\\
\tag{R.2} \; \; \;& \restr{\ddt\RZ\left(tZ\right)}{t=0}
= Z \text{ for all } Z\in T_X\iStkn,\label{eq:conditionC2}
\end{align*}
where $\RZ$ is the restriction of $\RX$ to $T_X\iStkn.$


In an embedded submanifold, retraction can be formulated based on matrix decomposition \cite[Prop. 4.1.2]{AbsiMS08} or distance  projection \cite[Prop. 5.54]{Boumal23}. In the case of indefinite Stiefel manifold, however, deriving such a retraction does not seem straightforward due to the generality of $A$ and $J$. Here, we will construct a retraction based on the Cayley transform for quadratic Lie group \cite[Chap. IV]{HairLW06}, which will be referred to as the \emph{Cayley retraction}. This type of retraction has been established for the generalized and the symplectic Stiefel manifolds, see \cite{WenY2013,GSAS21,ShusA2023,WangDPY2024} for details. To this end, for $X\in\iStkn$ and $Z\in T_{X}\iStkn,$ let us define the matrix
\begin{equation}\label{eq:2identicalSxz}
	S_{X,Z} = XJZ^{T}AXJX^{T} - XJZ^{T} +ZJX^{T}.
\end{equation}
Taking \eqref{eq:eq31a} into account, it can be rewritten as
\begin{align}\notag
	S_{X,Z} &=\dfrac{1}{2}XJZ^{T}AXJX^{T}- \dfrac{1}{2}XJX^{T}AZJX^{T}- XJZ^{T} +ZJX^{T}\\
	&= \left( I_n - \dfrac{1}{2}XJX^{T}A\right)ZJX^{T} - XJZ^{T}\left( I_n - \dfrac{1}{2}AXJX^{T}\right).
	\label{eq:2identicalSxz_2}
\end{align} 
By definition, $S_{X,Z}$ is skew-symmetric.
%
\begin{definition}\label{def:defRxzcay}
The Cayley retraction on $\iStkn$ at $X$ is the map 
\begin{equation}\label{eq:retractioncay}
	\begin{array}{rcl}
		\RZ^{\cay} : T_X\iStkn\, &\,\longrightarrow\, &\,\iStkn \\
		Z &\longmapsto& \RZ^{\cay}(Z) \coloneqq \left( I_n - \frac{1}{2}S_{X,Z}A\right)^{-1}\left( I_n + \frac{1}{2}S_{X,Z}A\right) X,
	\end{array}
\end{equation}
where $S_{X,Z}$ is as in \eqref{eq:2identicalSxz}. 
\end{definition} 
It is seen that $\RZ^{\cay}(Z)$ is defined  whenever $\left( I_n - \frac{1}{2}S_{X,Z}A\right)$ is invertible. When applied to an 
iterative optimization scheme, the update is performed 
in the direction of the tangent vector $Z\in T_{X}\iStkn$ along the curve 
\begin{equation}\label{eq:Cayley_curve}
Y^{\cay}(t,Z) \coloneqq \left(I_n - \dfrac{t}{2}S_{X,Z}A\right)^{-1}\left( I_n + \dfrac{t}{2}S_{X,Z}A\right) X =: \mathrm{\cay}\left( \dfrac{t}{2}S_{X,Z}A\right)X,
\end{equation}
for $t\geq 0$, on $\iStkn.$ The following proposition verifies Definition~\ref{def:defRxzcay} for the Cayley retraction. 
\begin{proposition}
		\begin{itemize}
\item[i)] For each $X\in\iStkn$, $\RZ^{\cay}$ in \eqref{eq:retractioncay} is well-defined, i.e.,  $\left(I_n-\frac{1}{2}S_{X,Z}A\right)^{-1}\left(I_n+\frac{1}{2}S_{X,Z}A\right)X \in \iStkn$ for any $Z\in T_X\iStkn$. 
\item[ii)] The map $\RX^{\cay}$ described in \eqref{eq:retractioncay} is a retraction.
\end{itemize}
\end{proposition}

\begin{proof}
i) Given $X\in\iStkn$ and $Z\in T_X\iStkn$, we need to show that $\left( \RZ^{\mathrm{cay}}(Z)\right)^TA\RZ^{\mathrm{cay}}(Z)=J. $ This equality is equivalent to
\begin{align*}
	&X^T(I_n-\tfrac{1}{2}S_{X,Z}A)^{-T}(I_n+\tfrac{1}{2}S_{X,Z}A)^TA(I_n+\tfrac{1}{2}S_{X,Z}A)(I_n-\tfrac{1}{2}S_{X,Z}A)^{-1}X
	&=J
\end{align*}
since  $\left(I_n+B\right)\left(I_n-B\right)^{-1}=\left(I_n-B\right)^{-1}\left(I_n+B\right)$ for any $B\in\Rbnn$ such that $I_n-B$ is invertible. This is in turn a necessary condition for
\[\left(I_n-\dfrac{1}{2}S_{X,Z}A\right)^{-T}\left(I_n+\dfrac{1}{2}S_{X,Z}A\right)^TA\left(I_n+\dfrac{1}{2}S_{X,Z}A\right)\left(I_n-\dfrac{1}{2}S_{X,Z}A\right)^{-1}=A,\]
or 
\begin{align*}
	A+\dfrac{1}{2}(S_{X,Z}A)^TA+&\dfrac{1}{2}AS_{X,Z}A+\dfrac{1}{4}(S_{X,Z}A)^TA(S_{X,Z}A)\\
	&=A-\dfrac{1}{2}(S_{X,Z}A)^TA-\dfrac{1}{2}AS_{X,Z}A+\dfrac{1}{4}(S_{X,Z}A)^TA(S_{X,Z}A).
\end{align*}
The last equality holds due to the symmetry of $A$ and the skew-symmetry of $S_{X,Z}$.

ii) Using \eqref{eq:2identicalSxz}, it is clear that $S_{X,0_X} = S_{X,0} = 0$ which implies that the condition \eqref{eq:conditionC1}  $\RZ^{\cay}(0)=X$ 
is fulfilled.

Next, we will verify the condition \eqref{eq:conditionC2} given $X\in\iStkn,Z\in T_X\iStkn,$ and $t\in\R$ such that $(I_n-\frac{t}{2}S_{X,Z}A)$ is invertible. Left-multiplying by $(I_n - \frac{t}{2}S_{X,\,Z}A)$ both sides of
\begin{equation*}
\RZ^{\mathrm{cay}}(tZ) = 		\left( I_n - \frac{t}{2}S_{X,\,Z}A\right)^{-1}\left( I_n + \frac{t}{2}S_{X,\,Z}A\right) X,
\end{equation*}
taking the derivative, and symplifying the result, we obtain
\begin{equation} \label{retractioncay3}
	\dfrac{\d}{\d t}\RZ^{\cay}(tZ)  =\left(I_{n}-\dfrac{t}{2}S_{X,Z}A\right)^{-1}\dfrac{1}{2}S_{X,Z}A\left(\RZ^{\cay}(tZ) + X\right).
\end{equation}
Specifically at $t=0,$ 
\begin{align*}
	&\restr{\ddt\RZ^{\mathrm{cay}}(tZ)}{t=0} = \dfrac{1}{2}S_{X,Z}A\left(\RZ^{\mathrm{cay}}(0)+X\right)=S_{X,Z}AX=Z
\end{align*}
which completes the proof.
\end{proof}

The well-definedness of the so-called \emph{Cayley curve} $Y^{\cay}(t,Z)$ in \eqref{eq:Cayley_curve} depends on the fact whether $2/t$ with $t>0$ is an eigenvalue of $S_{X,Z}A$. When $A = I_n,  J = I_k$, i.e., the Stiefel manifold case, 
since $S_{X,Z}A$ is skew-symmetric, it admits only pure imaginary eigenvalues. Therefore, $2/t$ cannot be an eigenvalue of $S_{X,Z}A$ for any $t > 0$. While \cite[Prop. 5.4]{GSAS21} provides a condition for the  well-definedness of the Cayley curve for the symplectic Stiefel manifold, we cannot find such a discussion for the 
generalized 
Stiefel manifold in the literature. Therefore, we present it here as a special case of the indefinite Stiefel manifold for completeness.
\begin{proposition}\label{prop:Cayleyretr_genStiefel}
	The \emph{Cayley curve} $Y^{\cay}(t,Z)$ defined in \eqref{eq:Cayley_curve} for the generalized Stiefel manifold, i.e., the manifold $\iSt_{A,J}$ with $A$ is spd and $J = 
	I_{k}$, is globally defined.
 \end{proposition}
\begin{proof}
	As above, we will generally show that the eigenvalues of the product matrix $SA$, where $S$ is skew-symmetric 
	and $A$ is spd, are either purely imaginary or zero. Indeed, the fact $\lambda\in \Cb$ is an eigenvalue of $SA$ means that there exists an  $x\in \Cbn, x\not= 0,$ such that $SAx = \lambda x$. 
	Applying the conjugate transpose to both sides and then multiplying from the right with $Ax$, we obtain that $x^*A(-S)Ax = \bar{\lambda}x^*Ax$. By assumption, it readily leads to $-\lambda \|x\|^2_A =\bar{\lambda} \|x\|^2_A$ where $\|\cdot\|^2_A$ is the $A$-norm defined by the spd matrix $A$. This equality implies that if $\lambda\not = 0$ then $\lambda + \bar{\lambda} = 0$ or $\lambda$ is purely imaginary. Hence, $SA$ does not admit $2/t$ with $t>0$ as one of its eigenvalues. 
\end{proof}

When $A$ is indefinite, the following counterexample tells us that the Cayley curve $Y^{\cay}(t,Z)$ is generally not globally defined. Indeed, consider the case when $A = \left[\begin{smallmatrix}
	-1&0\\0& 1
\end{smallmatrix}\right], J = -1$. For any $X = \left[\begin{smallmatrix}
	\pm\sqrt{\xi^2 + 1}\\\xi
\end{smallmatrix}\right]$ with $\xi\in \R$ and $Z  = \left[\begin{smallmatrix} \xi\\
	\pm\sqrt{\xi^2 + 1}
\end{smallmatrix}\right] \in T_X\iSt_{A,J}$. Calculation shows that $S_{X,Z}A = \left[\begin{smallmatrix}
	0&1\\1& 0
\end{smallmatrix}\right]$ which admits the eigenvalues $\pm 1$.

Although the Cayley curve is not defined globally, the following feature is enough for establishing the global convergence of the gradient descent algorithm in section~\ref{sec:Algorithm}.
\begin{lemma}\label{lem:Cayley_balldef}
	For any given $X\in \iSt_{A,J}$, there exists $\delta>0$ such that $\RZ^{\cay}(Z)$ is well-defined for all $Z\in T_X\iSt_{A,J}\ :\ \|Z\| < \delta$.
\end{lemma}
\begin{proof}
	In view of \cite[Lem. 2.3.3]{GoluV13}, it suffices to show that $\|S_{X,Z}A\| < 1$ with $\|Z\|<\delta$ for some $\delta>0$. We have the following simple estimate
	\begin{align*}
		\|S_{X,Z}A\| &= \|XJZ^TAXJX^TA - XJZ^TA + ZJX^TA\|\\
		& \leq \|Z\|\left(\|X\|^3\|J\|^2\|A\|^2 + 2\|X\|\|J\|\|A\|\right).
	\end{align*}
	Therefore, the statement holds for $\delta = 1/\left(\|X\|^3\|J\|^2\|A\|^2 + 2\|X\|\|J\|\|A\|\right)$.
\end{proof}

The main computational burden for the Cayley retraction $\RX^{\cay}$ on $\iStkn$ is contributed by 
the term $\left( I_n - \frac{1}{2}S_{X,Z}A\right)^{-1}$ which means solving a linear system of equations whose coefficient matrix is of size $n\times n.$ This task principally requires 
$O(n^3)$ flops.
In some cases, $k$ is considerably smaller than $n$ which makes that cost unreasonably high. It is therefore  expected to develop a more efficient alternative. In the same framework as in \cite{OviH23,BendZ21} for the symplectic Stiefel manifold and \cite{WangDPY2024} for the generalized Stiefel manifold, we obtain the following result. 

\begin{proposition}\label{prop:Rcaycomputation}
Given $X\in\iStkn$ and $Z\in T_{X}\iStkn,$  for any $t$ such that $\left( I_n - \frac{t}{2}S_{X,Z}A\right)$ is invertible, we have
\begin{equation}\label{eq:Cayley_econ}
	\RZ^{\cay}(tZ) = -X + \left(t\Lambda +2X \right)  \left( \dfrac{t^2}{4}\Lambda^{+}\Lambda - \dfrac{t}{2}M+I_k\right)^{-1},
\end{equation}
where $C^+: = JC^{T}A$ for any $C\in \Rbnk$, $M = X^{+}Z\in\Rbkk$ and $ \Lambda = Z -XM.$
\end{proposition}
\begin{proof}
Set $K \coloneqq \left[ \dfrac{1}{2}XM +\Lambda\;\; \;-X\right],\; N^{T} \coloneqq \begin{bmatrix}
	X^{+}\\
	Z^{+}\left( I_n - \dfrac{1}{2}XX^{+}\right) 
\end{bmatrix}. $

From Definition \ref{def:defRxzcay} and \eqref{eq:2identicalSxz_2}, we have
\begin{equation}
	\begin{aligned}\label{eq:eqKNt}
		KN^{T} &= \left[ \left(I_n -\dfrac{1}{2}XX^{+} \right)Z \;\;\;\;-X \right]\begin{bmatrix}
			X^{+}\\
			Z^{+}\left( I_n - \dfrac{1}{2}XX^{+}\right) 
		\end{bmatrix} \\
		& = \left( I_n - \dfrac{1}{2}XX^{+}\right)ZX^{+}-XZ^{+}\left( I_n - \dfrac{1}{2}XX^{+}\right)\\
		& = S_{X,Z}A, 
	\end{aligned}
\end{equation}
and
\begin{equation}
	\begin{aligned}\label{eq:eqNtK}
		N^{T}K &= \begin{bmatrix}
			\dfrac{1}{2}M\;\hspace{1cm}& -I_k\\ 
			& \\
			\dfrac{1}{2}Z^{+}\left(I_n - \dfrac{1}{2}XX^{+} \right)XM+ Z^{+}\left(I_n - \dfrac{1}{2}XX^{+} \right)\Lambda&\;\; \dfrac{1}{2}M
		\end{bmatrix}\\
		&= \begin{bmatrix}
			\dfrac{1}{2}M& -I_k\\
			Z^{+}\Lambda - \dfrac{1}{4}M^2&\;\;\dfrac{1}{2}M
		\end{bmatrix}= \begin{bmatrix}
			\dfrac{1}{2}M& -I_k\\
			\Lambda^{+}\Lambda - \dfrac{1}{4}M^2&\;\;\dfrac{1}{2}M
		\end{bmatrix}.
	\end{aligned}
\end{equation}
Here, to get \eqref{eq:eqKNt} and \eqref{eq:eqNtK}, we have used the facts that $X^{+}X =I_k,X^{+}\Lambda = 0$, and $X^{+}Z =- Z^{+}X.$ The last identity in \eqref{eq:eqNtK} is deduced from the equalities
$$\Lambda^{+}\Lambda = J(Z^{T} - M^{T}X^{T})A\Lambda =Z^{+}(I_n-XX^{+})\Lambda = Z^{+}\Lambda-Z^{+}XX^{+}\Lambda=Z^{+}\Lambda.$$

By \eqref{eq:retractioncay}, we have
\begin{align}\notag
	\RZ^{\cay}(tZ) &= \left( I_n - \dfrac{t}{2}S_{X,Z}A\right)^{-1}\left( I_n + \dfrac{t}{2}S_{X,Z}A\right) X \\ \notag
	& = \left( I_n - \dfrac{t}{2}KN^{T}\right)^{-1}\left( I_n + \dfrac{t}{2}KN^{T}\right) X \\ \notag
	& =  \left( I_n+\dfrac{t}{2}K\left( I_{2k} -\dfrac{t}{2}N^{T}K\right)^{-1} N^{T}\right)\left( I_n + \dfrac{t}{2}KN^{T}\right) X \\ 
	& = X + \dfrac{t}{2}KN^{T}X + \dfrac{t}{2}K\left(I_{2k} -\dfrac{t}{2}N^{T}K\right)^{-1}N^T\left(I_n + \dfrac{t}{2}KN^{T}\right)X\nonumber\\
	& = X + \dfrac{t}{2}KN^{T}X
	+ \dfrac{t}{2}K\left( I_{2k} - \dfrac{t}{2}N^{T}K\right)^{-1}\left( I_{2k} +\dfrac{t}{2}N^{T}K\right)N^{T}X\nonumber\\
	& = X + \dfrac{t}{2}KN^{T}X 
	- \dfrac{t}{2}K\left( I_{2k} - \dfrac{t}{2}N^{T}K\right)^{-1}\left( I_{2k} -\dfrac{t}{2}N^{T}K\right)N^{T}X\nonumber\\
	&\hspace{0.8cm}+tK\left( I_{2k} -\dfrac{t}{2}N^{T}K\right)^{-1}N^{T}X\nonumber\\
	& = X+tK\left( I_{2k} -\dfrac{t}{2}N^{T}K\right)^{-1}N^{T}X\nonumber\\
	& =  X+tK\left( I_{2k} -\dfrac{t}{2}N^{T}K\right)^{-1}\begin{bmatrix}
		I_k\\
		-\dfrac{1}{2}M
	\end{bmatrix}.\label{eq:Rcay2}
\end{align}
Here, the second equality follows 
from equation \eqref{eq:eqKNt} meanwhile the third is due to the Sherman-Morrison-Woodbury formula \cite[(2.1.4)]{GoluV13}.

Moreover, a direct inspection 
	using \eqref{eq:eqNtK} and the Schur complement of block matrices, see, e.g., \cite{Punta05}, yields that
\begin{align*}
	\left( I_{2k} -\dfrac{t}{2}N^{T}K\right)^{-1} &= \begin{bmatrix}
		I_{k} -\dfrac{t}{4}M&\dfrac{t}{2}I_k\\
		-\dfrac{t}{2}\left( \Lambda^{+}\Lambda-\dfrac{1}{4}M^2\right) &I_k-\dfrac{t}{4}M
	\end{bmatrix}^{-1}\\ 
	&= \begin{bmatrix}
		\Gamma^{-1}\left( I_k - \dfrac{t}{4}M\right)&-\dfrac{t}{2}\Gamma^{-1}\\
		\dfrac{2}{t}I_k - \dfrac{2}{t}\left( I_k - \dfrac{t}{4}M\right)\Gamma^{-1}\left( I_k - \dfrac{t}{4}M\right)&\left( I_k - \dfrac{t}{4}M\right)\Gamma^{-1}
	\end{bmatrix}
\end{align*}
where $\Gamma\coloneqq \left( \dfrac{t^2}{4}\Lambda^{+}\Lambda-\dfrac{t}{2}M+I_k\right)$. 
	Therefore, the second term 
	in equation \eqref{eq:Rcay2} becomes
	\begin{align*}
		tK&\begin{bmatrix}
			\Gamma^{-1}\left( I_k - \dfrac{t}{4}M\right)&-\dfrac{t}{2}\Gamma^{-1}\\
			\dfrac{2}{t}I_k - \dfrac{2}{t}\left( I_k - \dfrac{t}{4}M\right)\Gamma^{-1}\left( I_k - \dfrac{t}{4}M\right)&\left( I_k - \dfrac{t}{4}M\right)\Gamma^{-1}
		\end{bmatrix}\begin{bmatrix}
			I_k\\
			-\dfrac{1}{2}M
		\end{bmatrix}\\
		&=tK \begin{bmatrix}
			\Gamma^{-1}\left( I_k - \dfrac{t}{4}M\right)+\dfrac{t}{4}\Gamma^{-1}M\\
			\dfrac{2}{t}I_k - \dfrac{2}{t}\left( I_k - \dfrac{t}{4}M\right)\Gamma^{-1}\left( I_k - \dfrac{t}{4}M\right)-\dfrac{1}{2}\left( I_k - \dfrac{t}{4}M\right)\Gamma^{-1}M
		\end{bmatrix}\\
		&=tK\begin{bmatrix}
			\Gamma^{-1}\\
			\dfrac{2}{t}I_k - \dfrac{2}{t}\left( I_k - \dfrac{t}{4}M\right)\Gamma^{-1}\left[ \left( I_k - \dfrac{t}{4}M\right)+\dfrac{t}{4}M\right] \end{bmatrix}\\
		& =tK\begin{bmatrix}
			\Gamma^{-1}\\
			\dfrac{2}{t}I_k - \dfrac{2}{t}\left( I_k - \dfrac{t}{4}M\right)\Gamma^{-1} 
		\end{bmatrix}\\
		& = \dfrac{t}{2}XM\Gamma^{-1}+t\Lambda\Gamma^{-1} - 2X+2X\left( I_k - \dfrac{t}{4}M\right)\Gamma^{-1}\\
		& = -2X +(t\Lambda +2X)\Gamma^{-1}.
	\end{align*}
	Hence, \eqref{eq:Rcay2} reduces to
	\begin{equation*}
		\RZ^{\cay}(tZ)  = -X+(t\Lambda +2X)\Gamma^{-1} = -X +(t\Lambda +2X)\left( \dfrac{t^2}{4}\Lambda^{+}\Lambda - \dfrac{t}{2}M+I_k\right)^{-1}
	\end{equation*}
	which also completes the proof.
\end{proof}
\begin{remark}\label{rem:Compare_Cayley}
		The curve \eqref{eq:Cayley_curve} is a generalization of \cite[(7)]{WenY2013} for the Stiefel manifold and \cite[(3.9)]{ShusA2023} for the generalized Stiefel manifold; cf. \cite[(5.4)]{GSAS21} for the symplectic Stiefel manifold. The intermediate formula \eqref{eq:Rcay2} reduces to \cite[(3.5)]{WangDPY2024} and \cite[(19)]{WenY2013} with the corresponding restrictions on $A$ and $J$; cf. \cite[(5.7)]{GSAS21}. Finally, the most efficient formula \eqref{eq:Cayley_econ} which involves inversing matrix of size $k\times k$ only cannot be found in the literature for generalized Stiefel manifold; cf. \cite[(17)]{OviH23} and \cite[(5.2)]{BendZ21} for the case of symplectic Stiefel manifold. 
\end{remark}
	\begin{remark}
		The economical formula \eqref{eq:Cayley_econ} for the Cayley retraction can also be derived as a special case of Jiang-Dai's family of retraction proposed in \cite{JianD15} in the same framework as proceeded in \cite{OviH23}.
\end{remark}

To close this section, we discuss the possibility that the Cayley retraction \eqref{eq:retractioncay} is of second order. 
It can be interpreted from \cite[Def. 5.42]{Boumal23} that a retraction $\mathcal{R}$ is of second order if, for any $X$ and $Z\in T_X\iSt_{A,J}$, the second derivative of its curve $\frac{\d^2}{dt^2}(\RZ(tZ))$ belongs to the normal space  $T_X\iSt_{A,J}^{\perp}$. Second-order retraction helps compute the Riemannian Hessian of the cost function at any point without invoking the Riemannian connection \cite[Prop. 5.45]{Boumal23}. 
Differentiating the identity \eqref{retractioncay3} yields
$$
\dfrac{\d^2}{\d t^2}\RZ^{\cay}(tZ) = 2 \left(\left(I_{n}-\dfrac{t}{2}S_{X,Z}A\right)^{-1}\dfrac{1}{2}S_{X,Z}A\right)^2\left(\RZ^{\cay}(tZ) + X\right).
$$
Particularly at $t=0$, from \eqref{eq:2identicalSxz} and the fact $S_{X,Z}AX=Z$, 
we have
	\begin{align*}
		\left.\dfrac{\d^2}{\d t^2}\RZ^{\cay}(tZ)\right|_{t=0} &= S_{X,Z}AS_{X,Z}AX =S_{X,Z}AZ\\ 
		&=\left(XJZ^{T}AXJX^{T} - XJZ^{T} +ZJX^{T}\right)AZ.
	\end{align*}

With the example below, we point out that
$\left.\frac{\d^2}{\d t^2}\RZ^{\cay}(tZ)\right|_{t=0}$ is generally not a  normal vector at $X$ which in turn implies that the Cayley retraction is not of second order. Let us consider a point $X\in \iSt_{A,J}(2,3)$ with
\begin{equation*}
  A = \begin{bmatrix}\begin{array}{ccc}
  	\dfrac{-7}{3} & \dfrac{-2}{3} & \dfrac{4}{3}\\[0.8em]
  	\dfrac{-2}{3} & \dfrac{-23}{15} & \dfrac{-14}{15}\\[0.8em] 
  	\dfrac{4}{3} & \dfrac{-14}{15} & \dfrac{-2}{15}
	\end{array}
	\end{bmatrix},\; J=\begin{bmatrix}
	0&-1\\
	-1&0
	\end{bmatrix},\; X =  \begin{bmatrix}\begin{array}{cc}
		\dfrac{\sqrt{5}-\sqrt{3}}{6} & \dfrac{\sqrt{5}+\sqrt{3}}{6}\\[0.8em]
		\dfrac{\sqrt{5}+5\sqrt{3}}{30} & \dfrac{\sqrt{5}-5\sqrt{3}}{30}\\[0.8em] 
		-\dfrac{\sqrt{5}+5\sqrt{3}}{15} & \dfrac{-\sqrt{5}+5\sqrt{3}}{15} 
	\end{array}
	\end{bmatrix}.
\end{equation*}
It follows from \eqref{eq:eq31b} that, for $W= \left[\begin{smallmatrix}
	1&0\\
	0&-1
\end{smallmatrix}\right],\,X_{\perp}=\begin{bmatrix}
0&2&1
\end{bmatrix}^T, \text{ and } K=\begin{bmatrix}
1&0 \end{bmatrix},$ 
\begin{equation*}
	Z \coloneqq XW + A^{-1}X_{\perp}K =  \begin{bmatrix}\begin{array}{cc}
		\dfrac{\sqrt{5}-\sqrt{3}}{6} & -\dfrac{\sqrt{5}+\sqrt{3}}{6}\\[1em]
		\dfrac{\sqrt{5}+5\sqrt{3}-30}{30} & \dfrac{-\sqrt{5}+5\sqrt{3}}{30}\\[1em] 
		-\dfrac{2\sqrt{5}+10\sqrt{3}+15}{30}& \dfrac{\sqrt{5}-5\sqrt{3}}{15} 
	\end{array}
	\end{bmatrix}\in T_X\iSt_{A,J}.
	\end{equation*}
Therefore,
\begin{equation*}
	S_{X,Z} AZ= \begin{bmatrix}\begin{array}{cc}
			\dfrac{-3\sqrt{5}-7\sqrt{3}}{12} & \dfrac{\sqrt{5}+\sqrt{3}}{6}\\[1em]
			\dfrac{-3\sqrt{5}+35\sqrt{3}-60}{60} & \dfrac{\sqrt{5}-5\sqrt{3}}{30}\\[1em] 
			\dfrac{3\sqrt{5}-35\sqrt{3}-15}{30}& \dfrac{-\sqrt{5}+5\sqrt{3}}{15} 
		\end{array}
		\end{bmatrix}.
\end{equation*}
Since $JX^TS_{X,Z} AZ= \left[\begin{smallmatrix}
		2&-2/3\\
		-3/2&1/3
\end{smallmatrix}\right]$ is not symmetric, by contradiction it results in the fact that $S_{X,Z}AZ\neq AX\hat{W},$ for all $\hat{W}\in\symset(2).$ That is, in view of Proposition~\ref{prop:normalspace_gM}, 
$S_{X,Z}AZ$ does not belong to the normal space $T_X\iSt_{A,J}^{\perp}$ with respect to the Euclidean metric. 

\section{A gradient descent method}\label{sec:Algorithm}
As the necessary tools are available, we are now ready to present a Riemannian gradient descent method for solving the problem \eqref{eq:opt_prob}. Starting with an $X_0\in \iSt_{A,J}$, this method generates a sequence $\{X_j\}$ using a line search along the direction $-\gradm f(X_j)$ and updates the solution as 
$$X_{j+1} = \RZi^{\cay}(-\tau_j \gradm f(X_j)),$$ 
where $\tau_j >0$ is a carefully chosen step size. The choice for step size undoubtedly has a great impact on the method's performance. 
Here, we adopt a combination of a 
nonmonotone line search technique \cite{ZhanH2004} with the alternating Barzilai-Borwein \cite{BarB88}  which has been adapted to Riemannian optimization in \cite{IannP18} and simplified in \cite{HuLWY2020} for embedded submanifolds, see Algorithm~\ref{alg:non-monotone gradient} for details and \cite{GSAS21} for more choices of step size. One can see from \eqref{eq:nonmonotone_cond} that if $\alpha$ is chosen to be zero, this scheme reduces to the popular monotonically decreasing line search with Amijo's rule.

\begin{algorithm}[htbp]
	\caption{Riemannian gradient algorithm for the optimization problem~\eqref{eq:opt_prob}}
	\label{alg:non-monotone gradient}
	\begin{algorithmic}[1]
		\REQUIRE The cost function $f$, 
		initial guess $X_0\in\iSt_{A,J}$, $\gamma_0>0$,
		\mbox{$0<\gamma_{\min}<\gamma_{\max}$}, 
		$\beta, \delta\in(0,1)$, $\alpha \in [0,1]$, 
		$q_0=1$, $c_0 = f(X_0)$. 
		\ENSURE Sequence of iterates  $\{X_j\}$. 
		\FOR{$j=0,1,2,\dots$}
		\STATE Compute $Z_j = -\gradm f(X_j)$. 
		\IF{$j>0$} 
		\STATE{ 
			$
			\gamma_j=\left\{\begin{array}{ll}
				\dfrac{\langle W_{j-1},W_{j-1}\rangle}{\abs{\tr(W_{j-1}^T Y_{j-1}^{})}} 
				&\text{for odd } j, \\[5mm]
				\dfrac{\abs{\tr(W_{j-1}^T Y_{j-1}^{})}}{\langle Y_{j-1},Y_{j-1}\rangle}
				&\text{for even } j,
			\end{array}\right.
			$\\
			where $W_{j-1} = X_j - X_{j-1}$ and $Y_{j-1} =Z_j-Z_{j-1}$.
		}
		\ENDIF
		\STATE Calculate the trial step size $\gamma_j=\max\bigl(\gamma_{\min},\min(\gamma_j,\gamma_{\max})\bigr)$.			
		\STATE Find the smallest integer $\ell$  such that the 
		nonmonotone condition 
		\begin{equation}\label{eq:nonmonotone_cond}
			f\big(\RZi^{\cay}(
			\tau_j Z_j)\big) \le c_j  + \beta\, \tau_j\, \gM\big(\gradm f(X_j), Z_j\big)
		\end{equation} 
		holds, where $\tau_j=\gamma_j\, \delta^{\ell}$. 
		\STATE Set $X_{j+1} = \RZi^{\cay}(\tau_j Z_j)$.
		\STATE Update $q_{j+1} = \alpha q_{j} + 1$ and
		$\displaystyle{c_{j+1} = \frac{\alpha q_{j}}{q_{j+1}} c_{j}  + \frac{1}{q_{j+1}} f(X_{j+1})}$.
		\ENDFOR
	\end{algorithmic}
\end{algorithm}

Given Lemma~\ref{lem:Cayley_balldef}, according to \cite[Theorem~5.7]{GSAS21}, the global convergence of Algorithm~\ref{alg:non-monotone gradient} is guaranteed. Namely, regardless of the initial guess $X_0$, every accumulation point~$X_*$ of the sequence~
$\{X_j\}$
generated by Algorithm~\ref{alg:non-monotone gradient} is a~critical point of the cost function~$f$ in~\eqref{eq:opt_prob}, i.e., $\gradm f(X_*)=0$. The material presented in this section is theoretically known. Nevertheless, to the best of our knowledge, it is the first algorithmic development for Riemannian optimization on the indefinite Stiefel manifold. Despite of simplicity, 
the algorithm will prove its efficiency via numerical tests in the next section.

To finish this section, we give a rough estimate for the computational complexity of Algorithm~\ref{alg:non-monotone gradient}. To this end, we provide in Table~\ref{tab:ComplxEst} the complexity of its main steps. 
In the computation of the Riemannian gradient, if we choose the Euclidean metric, the term $\frac{4}{3}n^3$ will disappear, but it does not necessarily result in a shorter execution time since the total runtime also depends on the number of iterations \texttt{niter} and the number of the backtracking line searches \texttt{nls} in each of them. If $A$ is sparse, e.g., diagonal, the $4n^2k$ becomes $O(n)$. Meanwhile, the retraction formula \eqref{eq:Cayley_curve} always costs $O(n^3)$ and is obviously more expensive than \eqref{eq:Rcay2} and \eqref{eq:Cayley_econ}. From this table, one can give a rough statement on the cost of Algorithm~\ref{alg:non-monotone gradient}. For instance, using the retraction \eqref{eq:Cayley_econ}, the total cost is 
\begin{align}
\begin{split}
\label{eq:roughcost}
&\texttt{niter}\left[\left(\frac{4}{3}n^3 + 4n^2k +8nk^2 -2nk + 18k^3\right) + 6nk\right.\\ &\left. \qquad+ (8nk^2 -nk + 4k^3 - 4k^2)+ \texttt{nls}\left(nk^2 + 4nk +\frac{2}{3}k^3 + 4k^2 + \texttt{fval}\right)\right],
\end{split}
	\end{align}
where \texttt{fval} is the cost to evaluate the cost function.

\begin{table}[h!]
	\centering
	\caption{Computational complexity estimates for some main steps.\label{tab:ComplxEst}}
	\begin{tabular}{ p{6cm}  p{6cm} } 
		\toprule
		Computation step & complexity \\ \midrule 
		$\gradm f(X_j)$, given $\nabla f(X_j)$ & $\frac{4}{3}n^3 + 4n^2k +8nk^2 -2nk + 18k^3$ \\
		$\gamma_j$ in Algorithm~\ref{alg:non-monotone gradient}& $6nk$ \\
		Retraction \eqref{eq:Cayley_curve}: preparation & $2n^3 + 6n^2k +nk$\\
		Retraction \eqref{eq:Cayley_curve}: backtracking loop & $\frac{2}{3}n^3 + n^2k$\\
		Retraction \eqref{eq:Rcay2}: preparation & $10nk^2 -nk + 6k^3 -3k^2$ \\
		Retraction \eqref{eq:Rcay2}: backtracking loop & $2nk^2 + nk + \frac{28}{3}k^3 + 8k^2$ \\
		Retraction \eqref{eq:Cayley_econ}: preparation & $8nk^2 -nk + 4k^3 -4k^2$\\
		Retraction \eqref{eq:Cayley_econ}:  backtracking loop & $nk^2 + 4nk +\frac{2}{3}k^3 + 4k^2$\\
		
		\bottomrule
	\end{tabular}
\end{table}


\section{Numerical examples}\label{sec:NumerExam}
In this section, we run several examples to verify the theoretical results obtained previously and to better understand the overall numerical behavior of the method. First, the parameters in Algorithm~\ref{alg:non-monotone gradient} need to be set up. Most of them are fixed for all examples. 
Namely, for the backtracking search, we set $\beta = 1\e{-4}$, $\delta = 5\e{-1}$, $\gamma_0 = 1\e{-3}$, 
$\gamma_{\min} = 1\e{-15}$, $\gamma_{\max} =1\e{+5}$.  Moreover, $\alpha = 0.85$ is used in the 
nonmonotone condition. 
The stopping tolerance $\texttt{rstop}$, 
which indicates a convergence when $\|\gradm f(X_j)\|_{X_j}\leq\texttt{rstop}\|\gradm f(X_0)\|_{X_0}$, is usually set to be either $1\e{-9}$ or $1\e{-6}$ upon example.

For every test, 
if the Euclidean Hessian of the extended cost function $\nabla^2 \bar{f}(X)$ is positive-definite, 
we will take it to define the tractable metric as an adoption of the suggestion in \cite{DongGGG2022,GaoPY2024}. 
The hint from \cite{ShusA2023} 
employing $A$ is not applied here because $A$ in the context of this paper is generally not spd. 
We will examine three schemes for the Cayley retraction to see the difference in efficiency and feasibility. We will also compare our proposed method with the augmented Lagrangian method (ALM) \cite[Framework 17.3]{NoceW06}. 
This method has recently been reported to deliver good results when applied to the optimization problem under the symplecticity constraint \cite{OviedoLara2025}. Specifically applied to our present problem, at step $j$ of the main iteration, the method approximately solves the subproblem with the stopping criterion $\|X_j^TAX_j-J\|_F/(j + 10)$ using the spectral gradient method as in \cite{OviedoLara2025}\footnote{In fact, we adapt the MATLAB code, provided by the authors at \href{https://www.mathworks.com/matlabcentral/fileexchange/181008}{https://www.mathworks.com/matlabcentral/fileexchange/181008}, to our problem.}. After this step, we update the multiplier and penalty parameter as $\Lambda_{j+1} = \Lambda_{j}+\mu_j(X_{j+1}^TAX_{j+1}-J)
$ and
$$
\mu_{j+1} = \left\{ \begin{array}{ll}
	&\mu_j,\mbox{ if } \|X_{j+1}^TAX_{j+1}-J\|_F\leq \tau\|X_j^TAX_j-J\|_F\mbox{ or } \mu_j\geq 1\e+14,\\
	&\rho\mu_{j},
\end{array}\right.
$$
respectively, where $\tau\in (0,1), \rho > 1$. Usually, we initialize the parameters as: Lagrange multiplier $\Lambda_0 = 0, \mu_0 = 1$, and $X_0$ is the normalized matrix of 1's and set $\tau = 0.25, \rho = 10$ as recommended in  \cite[sect. 10.4]{SunYuan2006}. We stop the main iteration when the Euclidean norm of the associated Lagrangian function meets the relative convergence condition, specified by $\texttt{rstop}_{ALM}$, in the same manner as the proposed Riemannian gradient descent (RGD) method and the feasibility violation does not exceed the tolerance \texttt{feastol}. 
Additionally, we also point out that the step size strategy described in Algorithm~\ref{alg:non-monotone gradient} is helpful by comparing the result with that using a fixed step size. 
The numerical experiments are performed on a laptop with Intel(R) Core(TM) i7-4500U (at 1.8GHz, 5MB Cache, 8GB RAM) running MATLAB R2023a under Windows 10 Home.  The main code is made available at \href{https://sites.google.com/view/ntson/code}{https://sites.google.com/view/ntson/code}.



\subsection{Trace minimization for positive-definite matrix pencil}\label{ssec:Example_Tracemin}
We consider in this section the minimization 
problem \eqref{eq:opt_prob} where 
$f(X) = \tr\left(X^TMX\right)$, \mbox{$M\in \SPD(n)$}, $A\in \symset(n)$ is nonsingular with  $\irm_+(A)=p, \irm_-(A) = m$, and  
$p+m = n$, $J = \diag(I_{k_p},-I_{k_m})$ and $k_p+k_m = k, k_p\leq p, k_m \leq m$. 
In view of \cite[Thm. 3.5]{KovaV95}, recalling \eqref{eq:KovacStrikoVeselic}, a global  solution $X_*$ can be used to compute $k_p$ positive and $k_m$ negative eigenvalues having the smallest magnitude of the matrix pencil $(M,A)$. 
Indeed, it follows that $MX_* = AX_*\Lambda$ where $\Lambda = \diag(\Lambda_+,\Lambda_-)$ is determined as in the theorem mentioned. This equality leads to $X_*^TMX_* = \diag(\Lambda_+,-\Lambda_-)$. Finally, the sought eigenvalues can be determined as that of the $k_p\times k_p$ matrix $\Lambda_+$ and $k_m\times k_m$ matrix $\Lambda_-$. In each test, we present the value of the objective function (obj.), the norm in the chosen metric of the Riemannian gradient (grad.), the relative error in computing the mentioned eigenvalues (eig. rel. err.) 
$\|MV - AVD\|_F/\|AVD\|_F$ in which 
$V$ is a matrix of (generalized) eigenvectors associated with the computed eigenvalues 
listed in the diagonal of the diagonal matrix $D$, the feasibility (feas.) $\|X_*^TAX_*-J\|_F$, the number of iterations ($\#$iter.), the number of function evaluations ($\#$eval.), and the CPU time (CPU) needed to obtain $X_*$. 

In the first context, we choose the Lehmer matrix of size $n = 200$ as the matrix $M$, and set $A = \diag(1,\ldots,p,-m,\ldots,-1)$  
with $p=150$ and $m=50$. 
We consider both Euclidean and weighted Euclidean metrics for two triplets $(k,k_p,k_m) = (5,3,2)$ and $(20,15,5)$ which results in four instances. With \texttt{rstop}$=1\e-9$, the obtained numerical results, noted as RGD, are given in Table~\ref{tab:tracemin_500}. It is obvious to see that using an appropriately chosen metric can significantly speed up the iterative scheme. Surprisingly, it also helps 
reduce the error of eigenvalue and eigenvector computations. In the forthcoming tests, we only use an individually chosen weighted Euclidean metric.  We also run the ALM for this model with $\texttt{rstop}_{ALM} = 1\e-7$ and $\texttt{feastol} = 1\e-10$. This specific convergence condition and the maximal numbers of the main and inner iterations, 10 and 5000, respectively, are set so that the accuracy of the obtained result can be comparable to that of the RGD method. Note that the norm of the Riemannian gradient of the cost function, which is indeed quite the same as the norm of the gradient of the Lagrangian function, is displayed to make the comparison appropriate.

\begin{table}[h!]
	\centering
	\caption{Trace minimization tests for positive-definite matrix pencils $(M,A)$ with $n=200$, $p = 150$, $m = 50$,  $M=\emph{\texttt{gallery}}(\emph{'\texttt{lehmer}'},n),\, A =\diag(1,\ldots,p,-m,\ldots,-1)$,  $\emph{\texttt{rstop}} = 1\e-9$ for the RGD method and $\texttt{rstop}_{ALM} = 1\e-7$ and $\texttt{feastol} = 1\e-10$ for the ALM method.\label{tab:tracemin_500}}
	\footnotesize
	\begin{tabular}{c l c c c c r r r}
		\toprule
	Meth.& retr.&obj.&grad.&eig. rel. err.&feas.&$\#$iter.& $\#$eval.& CPU \\ \midrule 
		&\multicolumn{8}{c}{$k = 5$, $k_p = 3$, $k_m= 2$, $\Mx = I_n$}\\
		&\eqref{eq:Cayley_curve}&$2.244\e-4$ &$3.823\e-9$ &$2.164\e-6$ &$2\e-12$& 13932& 15016& 16.7 \\
		RGD&\eqref{eq:Rcay2}&$2.244\e-4$&$2.134\e-9$ &$1.208\e-6$&$9\e-14$& 12463 &13510 &3.3 \\
		&\eqref{eq:Cayley_econ}&$2.244\e-4$& $4.458\e-9$ &$2.524\e-6$&$1\e-12$ & 10824& 11657& 2.3 \\
		ALM & &$2.294\e-4$ & $2.021\e-5$ & $1.135\e-2$ & $3\e-06$ & $4$ &$20695$& $1.1$\\
		\midrule 
		&\multicolumn{8}{c}{$k = 20$, $k_p = 15$, $k_m= 5$, $\Mx = I_n$}\\
		&\eqref{eq:Cayley_curve}&$9.084\e-4$ &$7.881\e-9$ &$2.067\e-6$&$5\e-12$& 17122& 18493 &30.8 \\
		RGD&\eqref{eq:Rcay2}&$9.084\e-4$& $7.618\e-9$ &$1.997\e-6$& $7\e-13$& 17649& 19072 &14.7 \\
		&\eqref{eq:Cayley_econ}&$9.084\e-4$& $7.920\e-9$ &$2.078\e-6$&$7\e-12$ &16248& 17517& 11.6 \\
		ALM & & $9.130\e-4$ & $5.848\e-5$ & $1.510\e-2$ & $1\e-09$ & $10$ & $51020$ & $7.1$\\
		\midrule 
		&\multicolumn{8}{c}{$k = 5$, $k_p = 3$, $k_m= 2$, $\Mx = M$}\\
		&\eqref{eq:Cayley_curve}& $2.244\e-4$ &$8.405\e-10$& $8.207\e-8$ &$9\e-15$& 92& 93&0.3 \\
		RGD&\eqref{eq:Rcay2}&$2.244\e-4$& $1.258\e-9$& $4.397\e-8$& $2\e-13$& 103& 104& 0.3\\
		&\eqref{eq:Cayley_econ}&$2.244\e-4$ &$1.211\e-9$& $1.215\e-7$& $2\e-13$& 97&98& 0.2\\
		\midrule
		&\multicolumn{8}{c}{$k = 20$, $k_p = 15$, $k_m= 5$, $\Mx = M$}\\
		&\eqref{eq:Cayley_curve}& $9.084\e-4$& $1.173\e-9$& $3.676\e-8$ & $2\e-14$& 109 & 110 & 0.5 \\
		RGD&\eqref{eq:Rcay2}& $9.084\e-4$& $1.317\e-9$ & $5.851\e-9$& $1\e-12$& 127& 128& 0.3 \\
	&	\eqref{eq:Cayley_econ}& $9.084\e-4$& $9.063\e-10$& $1.350\e-8$& $1\e-12$ & 121& 122&0.3 \\
		\bottomrule
	\end{tabular}
\end{table}

We continue with this model to check the scalability of the proposed algorithm. To this end, we let $n=200, 400,\ldots,6400$, $p = 3n/4, m = n/4$, while we fix $k=10, k_p=k_m=5$. In Table~\ref{tab:scalability}, we report some measures versus the problem size. One can observe that while the gradient norm and the feasibility violation are quite stable, the number of iterations as well as the function evaluations increase predictably. Specifically, the CPU time indicates that the computational complexity of Algorithm~\ref{alg:non-monotone gradient} is approximately $O(n^3)$, which basically matches the estimate given in \eqref{eq:roughcost}.
\begin{table}[h!]
	\centering
	\caption{Scalability check of Algorithm~\ref{alg:non-monotone gradient} using formula \eqref{eq:Cayley_econ} with the trace minimization for $n\times n$ positive-definite matrix pencils $(M,A)$ where $M=\emph{\texttt{gallery}}(\emph{'\texttt{lehmer}'},n)$, $ A =\diag(1,\ldots,p,-m,\ldots,-1)$, $J = \diag(I_5,-I_5)$ with $p = 3n/4$, $m = n/4$,  $\emph{\texttt{rstop}} = 1\e-8$.\label{tab:scalability}}
	\footnotesize
	\begin{tabular}{l c c c c c c}
		\toprule
		$n$ & 200 & 400 & 800 & 1600 & 3200 & 6400 \\ \midrule 
		grad.& $1.3\e-8$ & $1.1\e-8$ & $1.1\e-8$ & $1.1\e-8$ & $6.7\e-9$ & $1.2\e-8$ \\
		eig. rel. err.& $3.2\e-7$ & $5.7\e-7$ & $1.9\e-7$ & $2.7\e-7$ & $2.6\e-7$ & $2.5\e-2$ \\
		feas.& $1\e-11$ & $5\e-12$ & $1\e-10$ & $2\e-12$ & $1\e-11$ & $6\e-12$ \\
		$\#$iter.& 108 & 120 & 192 & 201 & 233 &251\\
		$\#$eval. & 109 & 121 & 194 & 205 & 248 & 262\\
		CPU & 0.30 & 0.87 & 5.24 & 28.19 & 201.98 & 1376.28\\
		\bottomrule
	\end{tabular}
\end{table}

Next, we numerically examine the constant step size strategy for the proposed algorithm using the same model problem with $n = 5000$ and $\texttt{rstop} = 1\e-6$. For unconstrained optimization of quadratic functions, it was confirmed in \cite[chap. 1]{Bertsekas1999} that a convergence to a solution is guaranteed for the gradient descent method if the fixed step size is smaller than $2/\lambda_{\max}$, where $\lambda_{\max}$ is the largest eigenvalue of the Hessian. Our situation is a bit more involved, but this hint can provide a range for the tests. Since the recommended upper bound in our cases is approximately $7\e-4$, we run the RGD method with fixed step sizes, namely, $5\e-6, 1\e-5, 5\e-5, 1\e-4, 5\e-4, 1\e-3, 5\e-3$ and compare the results with our Algorithm~\ref{alg:non-monotone gradient}. As expected, only Algorithm~\ref{alg:non-monotone gradient} delivers a convergence after 242 iterations. We report the running history in Figure~\ref{fig:fixstep}. One can see that, except for a slight loss of feasibility and about 1.5 times slower per iteration, the non-monotone line search strategy in  Algorithm~\ref{alg:non-monotone gradient} is still much more efficient than a fixed step size in reaching a solution.

\begin{figure}[t]
	\centering
	\includegraphics[width=\textwidth]{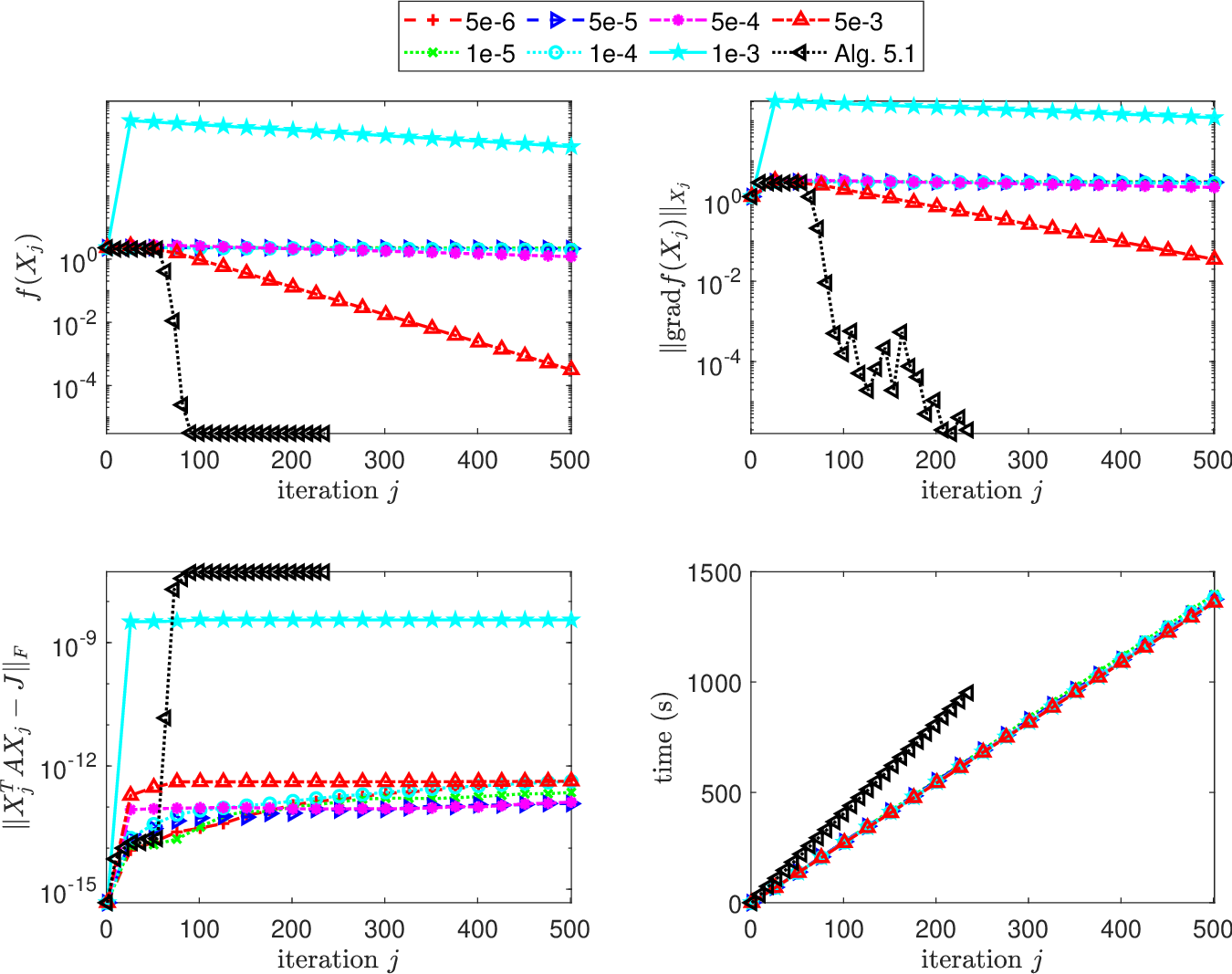}
	{\caption{Comparison of the step size strategy used in Algorithm~\ref{alg:non-monotone gradient} with different fixed step sizes (specified on the top) on the model problem described in Table~\ref{tab:scalability} with $n=5000$ and $\texttt{rstop} = 1\e-6$.}	\label{fig:fixstep}}
\end{figure}

Now, we proceed with different positive-definite matrices in the MATLAB gallery with a similar setting. Namely, we consider the cost function  \eqref{eq:opt_prob} determined by different matrices $M$ with $(n,p,m) = (2000,1000,1000)$ and $(k,k_p,k_m)=(10,5,5)$, $A = \diag(1,\ldots,p,-1,\ldots,-m)$ and $\texttt{rstop}=1\e-9$. For ALM, we set $\texttt{rstop}_{ALM} = 1\e-7$, $\texttt{feastol} = 1\e-8$, with maximal 20 and 1000 iterations for the main and inner loops, respectively.  
More details and the obtained results are described in Table~\ref{tab:tracemin_2000}. 
One can easily see that the Cayley retraction formulae \eqref{eq:Rcay2} and \eqref{eq:Cayley_econ} are faster in yielding the solution, while the conventional one \eqref{eq:Cayley_curve} is always the best in preserving the feasibility. Compared to RGD, the ALM method delivers less acurate result, especially the gradient norm and the relative error in eigenvalue computation, although it consumes more time. Particularly, it does not yield a solution in the case of the $\texttt{tridiag}$ matrix.

\begin{table}[h!]
	\centering
	\caption{Trace minimization tests for positive-definite matrix pencils $(M,A)$ with $n=2000$, $p = 1000$, $m = 1000$, $k = 10$, $k_p = 5$, $k_m= 5$, $\Mx = M$, $A =\diag(1,\ldots,p,-1,\ldots,-m)$,  $\emph{\texttt{rstop}} = 1\e-9$ for the RGD method and $\texttt{rstop}_{ALM} = 1\e-7$ and $\texttt{feastol} = 1\e-8$ for the ALM method.\label{tab:tracemin_2000}}
	\footnotesize
	\begin{tabular}{c l c c c c r r r}
		\toprule
		Meth.&retr.&obj.&grad.&eig. rel.  err.&feas.&$\#$iter.& $\#$eval.& CPU \\ \midrule 
		&\multicolumn{8}{c}{$M=\texttt{gallery}(\texttt{'lehmer'},n)$}\\
		&\eqref{eq:Cayley_curve}&$3.940\e-6$ &$2.717\e-9$ &$7.282\e-7$ &$2\e-14$& 166 & 174 & 90.7 \\
		RGD&\eqref{eq:Rcay2}& $3.940\e-6$ & $1.146\e-9$ & $7.228\e-5$& $2\e-12$& 163 &171 & 40.3 \\
		&\eqref{eq:Cayley_econ}&$3.940\e-6$&$2.920\e-9$ &$2.392\e-6$&$3\e-09$ & 165 & 173 & 41.2 \\
		ALM&&$2.048\e-5$ & $6.443\e-3$ & $9.023\e-2$ & $7\e-16$ & $20$ & $21163$ & $257.5$\\
		\midrule 
		&\multicolumn{8}{c}{$M$ = \texttt{gallery}('\texttt{gcdmat}',$\,n$)}\\
		&\eqref{eq:Cayley_curve}& $5.223\e-0$ &$1.934\e-8$& $2.721\e-8$ &$8\e-14$&513&530& 262.4 \\
		RGD&\eqref{eq:Rcay2}& $5.223\e-0$ & $3.595\e-8$ &$1.327\e-8$& $5\e-11$ & 559& 564 &126.5\\
		&\eqref{eq:Cayley_econ}& $5.223\e-0$ &$2.868\e-8$& $5.681\e-11$ &$5\e-11$& 633& 666& 144.4\\
		ALM& & $5.228\e-0$ & $8.118\e-2$ &$3.292\e-4$ &$9\e-12$ & 20 &19499 &232.7\\
		\midrule
		&\multicolumn{8}{c}{$M$ = \texttt{gallery}('\texttt{mohler}',$\, n,\ 0.5$)}\\
		&\eqref{eq:Cayley_curve}& $5.716\e-3$ &$2.130\e-8$ &$1.531\e-6$ & $2\e-14$& 260 & 296& 137.1 \\
		RGD&\eqref{eq:Rcay2}& $5.716\e-3$& $2.393\e-8$ & $7.323\e-5$& $7\e-12$& 328& 450 & 75.3 \\
		&\eqref{eq:Cayley_econ}& $5.716\e-3$& $2.337\e-8$& $9.779\e-6$& $5\e-12$& 248& 291& 56.9\\ 
		ALM& &$6.322\e-3$ & $1.613\e-2$ & $1.146\e-1$ &$1\e-13$ & 20 & 22311 &266.1\\
		\midrule
		&\multicolumn{8}{c}{$M$ = \texttt{gallery}('\texttt{minij}', $n$)}\\
		&\eqref{eq:Cayley_curve}& $2.585\e-3$ &$8.918\e-8$ &$1.338\e-6$ & $3\e-14$& 203 & 237 & 111.1 \\
		RGD&\eqref{eq:Rcay2}& $2.585\e-3$& $4.215\e-8$ & $2.482\e-6$& $2\e-08$& 199& 225 & 48.4 \\
		&\eqref{eq:Cayley_econ}& $2.585\e-3$& $7.641\e-8$ & $2.656\e-6$& $3\e-10$& 202& 233& 48.4 \\ 
		ALM& &$4.033\e-3$ &$4.164\e-2$ &$4.368\e-1$ &$3\e-13$ &20 &21725 &261.2\\
		\midrule
		&\multicolumn{8}{c}{$M$ = \texttt{gallery}('\texttt{tridiag}', $n$)}\\
		&\eqref{eq:Cayley_curve}& $2.039\e-6$ &$6.377\e-10$ &$2.703\e-5$ & $1\e-14$& 39 & 54& 12.8 \\
		RGD&\eqref{eq:Rcay2}& $2.039\e-6$& $6.940\e-10$ & $2.702\e-5$& $4\e-13$& 39& 54 & 0.1 \\
		&\eqref{eq:Cayley_econ}& $2.039\e-6$& $6.604\e-10$ & $2.702\e-5$& $4\e-13$& 39& 54& 0.1 \\ 
		ALM& \multicolumn{8}{c}{failed}\\
		\bottomrule
	\end{tabular}
\end{table}

Now, we move our attention to the LREVP in terms of the minimization problem  \eqref{eq:IndefTraceMin}. Here, we use data of larger sizes, stemming from practice. In the first test, following \cite[sect. 6]{BaiL14}, we make use of the University of Florida sparse matrix collection \cite{DavisHu2011}. Namely, the matrix \texttt{bcsstk21}, which is of size $3600\times 3600$, and the leading principal submatrix of the same size of the matrix \texttt{sts4098} are assigned to $K$ and $M$, respectively. The block structure makes the size of the problem 7200. We aim at computing the smallest four positive eigenvalues, and therefore the matrix $J$ is chosen as $I_4$. For the RGD method, parameter setting is as before, while an initial guess $X_0$ can be simply chosen as $\left[\begin{smallmatrix}
		\tilde{V}\\\tilde{V}
	\end{smallmatrix}\right]$, where $\tilde{V}$ is any $n\times k$ random matrix with orthonormal columns 
	divided by $\sqrt{2}$. In our test, we start with an $n\times k$ random matrix, preset by \texttt{rng}(1,'\texttt{philox}'), and then 
	orthogonalize the obtained matrix. With $\texttt{rstop} = 1\e-9$, the convergence is reached after 45 iterations. The condition $\texttt{rstop} = 1\e-11$ is fulfilled at the 108-th iteration. 
	 Meanwhile, after trying various sets of paramters and initial guesses,  we observe that ALM helps reduce the norm of the associated Lagrangian function very fast at the beginning and usually stagnates when it reaches $1\e+1$, which is too far from yielding a reliable approximation for the eigenvalues. Therefore, in Figure~\ref{fig:Florida} and Table~\ref{tab:Florida}, we report the results obtained using the proposed RGD only. One can see that, with the tolerance $\texttt{rstop}=1\e-11$, our method can deliver a comparably accurate solution as in \cite{BaiL14}\footnote{Note that \cite{BaiL14} considered the generalized LREVP including the LREVP as a special case.}

\begin{table}[h!]
	\centering
	\caption{A LREVP in terms of the trace minimization for the positive-definite matrix pencils $(Q,A)$ with $Q = \diag(K,M)$, $K, M \in \SPD(p)$ derived from the University of Florida sparse matrix collection, $A =\left[\begin{smallmatrix}
			0&I_p\\I_p&0
		\end{smallmatrix}\right]$ with $p = 3600$, $J=I_k$ with $k = 4$, and $\Mx = Q$,.\label{tab:Florida}}
	\footnotesize
	\begin{tabular}{c c c c c c r r r}
		\toprule
		Retr.&\texttt{rstop}.&obj.&grad.&eig. rel. err.& feas.&$\#$iter.& $\#$eval.& CPU \\ \midrule 
		\eqref{eq:Cayley_curve}&$1\e-09$&123.9778&$1.3520\e-5$ &$7.766\e-07$ &$4\e-15$ &45 &46 &393.43\\
		\eqref{eq:Rcay2}&$1\e-09$&123.9778 &$1.3537\e-5$ &$7.766\e-07$& $8\e-14$&45 & 46&2.91\\
		\eqref{eq:Cayley_econ}&$1\e-09$&123.9778&$1.3540\e-5$&$7.766\e-07$ &$6\e-14$& 45 & 46 & 2.32 \\
		\eqref{eq:Cayley_econ}&$1\e-11$ & 123.9778 & $6.1818\e-8$ & $8.135\e-13$& $1\e-13$& 108 & 109 & 5.39 \\
		\bottomrule
	\end{tabular}
\end{table}

\begin{figure}[t]
	\centering
	\includegraphics[width=\textwidth]{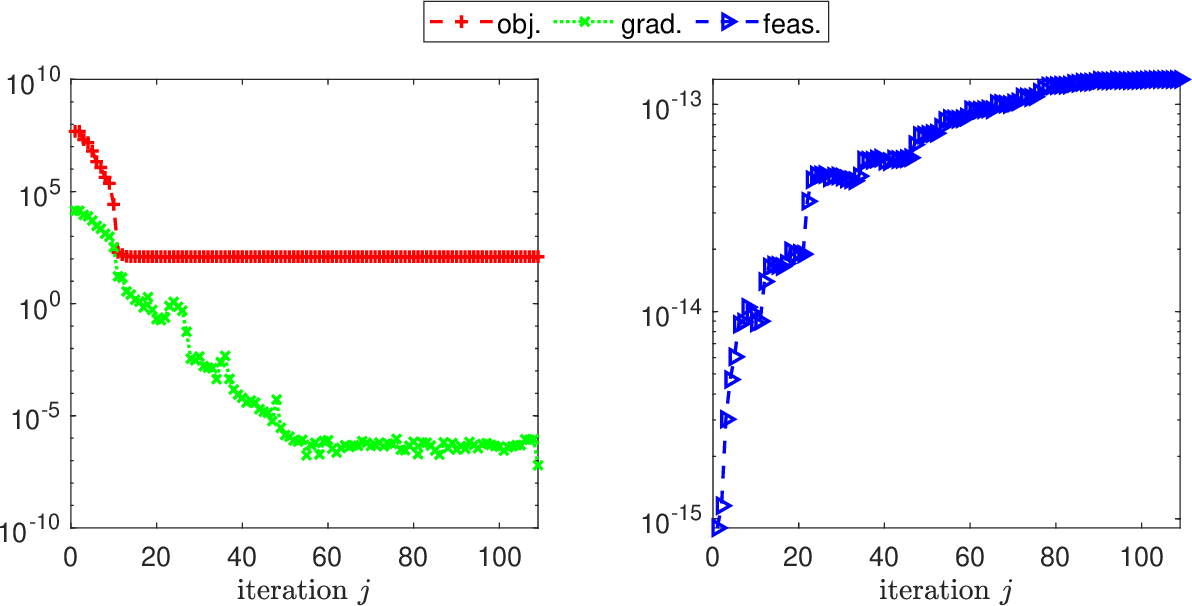}
	{\caption{Convergence history of the LREVP with data specified in Table~\ref{tab:Florida}.}	\label{fig:Florida}}
\end{figure}

Another model for the LREVP is considered next where the matrices $K$ and $M$ are  derived from electronic structure calculations, which were also used in \cite{KresPS2014,Pand2019,BaiL13}. 
Here, for the problem of size $n=11320$, the stopping tolerance \texttt{rstop} is set as $1\e-6$ 
while other parameters are the same as before.  
Results are given in Table~\ref{tab:traceminKM}. Note however that our use of this example is merely as a model with practical data; the eigenvalue computation scheme proposed in \cite{KresPS2014} is apparently more efficient. Indeed, a test using the retraction \eqref{eq:Cayley_curve} with \texttt{rstop} $=1\e-8$ delivers four approximate eigenvalues that can match about the first ten digits of the reference eigenvalues
\begin{align*}
	&\lambda^+_1 \approx 0.541812517132466,\quad \lambda^+_2 \approx 0.541812517132473,\\
	&\lambda^+_3 \approx 0.541812517132498,\quad \lambda^+_4 \approx 0.615143209274579,
\end{align*}
which is a bit worse than the approximation reported in \cite[Fig. 1]{KresPS2014}.
\begin{table}[h!]
	\centering
	\caption{A LREVP via trace minimization for the positive-definite matrix pencil $(Q,A)$ with $Q = \diag(K,M)$, $K, M \in \SPD(p)$ from the electronic structure calculation, $A =\left[\begin{smallmatrix}
			0&I_p\\I_p&0
		\end{smallmatrix}\right]$ with $p = 5560$, $J=I_k$ with $k =4$, $\Mx = Q$, $\emph{\texttt{rstop}} = 1\e-6$.\label{tab:traceminKM}}
	\footnotesize
	\begin{tabular}{l c c c c c r}
		\toprule
		Retr.&obj.&grad.&feas.&$\#$iter.& $\#$eval.& CPU \\ \midrule 
		\eqref{eq:Cayley_curve}&2.240580760678145&$1.820\e-6$ &$1\e-14$& 89 & 90 & 3679 \\
		\eqref{eq:Rcay2}& 2.240580760754902 & $1.820\e-6$ & $3\e-09$& 89 &90 & 1170 \\
		\eqref{eq:Cayley_econ}&2.240580760754425&$1.821\e-6$ & $3\e-09$ & 89 & 90 & 1201  \\
		\bottomrule
	\end{tabular}
\end{table}

\subsection{Matrix least square problems}\label{ssec:Example_MatEq}
We consider minimization of the cost function of the form
\begin{equation}\label{eq:fitting}
	f(X) = \|GX-B\|^2_F.
\end{equation}
This form is an approach for various problems. First, with suitable specification, 
it is the well known Procrustes problem \cite{Scho1966,High1988,AndeE1997,BojaL1999,BhatC2019}. In this subsection, to intimate the conventional situation \cite{Scho1966}, we examine the problem on the $J$-orthogonal group. That is, given $G, B\in \Rbln,\, l\geq n$, find a matrix $X\in \iSt_{J,J}$ with $J = \diag(\pm 1,\ldots,\pm 1) \in \Rbnn$ that minimizes $f$ defined in \eqref{eq:fitting}. It is worth noting that in a recent work \cite{HeYWX2024}, a block coordinate descent method was developed. To set up, let us consider $J = (I_p,-I_m)$ with $p+m=n$. Then $G$ is randomly generated with \texttt{rng default} in MATLAB. A global solution 
is prescribed as 
$V = \diag(V_1,V_2),$ 
where $V_1\in \O(p), V_2\in \O(m)$ are  initialized randomly with \texttt{rng}(1,'\texttt{twister}') and then orthogonalized by \texttt{orth}. This results in the setting that $B = GV$. The identity matrix $I_n$ is chosen as the initial guess and $\texttt{rstop} = 1\e-6$ for the RGD method; naturally, only the Cayley retraction \eqref{eq:Cayley_curve} is used. For the ALM method, we set $\texttt{rstop}_{ALM} = 1\e-9$ and $\texttt{feastol} = 1\e-12$ while the maximal number of the main and the subproblem iterations are 20 and 100, respectively. Table~\ref{tab:Procrustes} displays the results for different sizes of problem with $p = 3n/4$ in which diff. represents the difference between the prescribed and the numerically obtained solutions. Apparently, the iterate generated by the RGD method converges to the prescribed global minimizer only in the cases $l=n=1000$ and $(l,n)=(1500,1000)$; this is normal since the problem can have more than one global solution and/or the obtained solution might not be a global one. Overall, the generated iterate converges as expected. Surprisingly, the ALM method offers better results in the case of $(l,n)=(1000,500)$.
\begin{table}[h!]
	\centering
	\caption{Procrustes problem $\|GX - B\|_F^2$ on the $J$-orthogonal group tests with $\Mx = G^TG$, $J = \diag(I_p,-I_m)$,   $\emph{\texttt{rstop}} = 1\e-6$ for the RGD method and $\texttt{rstop}_{ALM} = 1\e-9$ and $\texttt{feastol} = 1\e-12$ for the ALM method.\label{tab:Procrustes}}
	\footnotesize
	\begin{tabular}{c c c c c r r r}
		\toprule
		Meth. & obj.&grad.&diff.&feas.&$\#$iter.& $\#$eval.& CPU \\  \midrule 
		\multicolumn{8}{c}{$l=500, n = 500$}\\
		RGD&$1.2912\e+01$ &$6.9214\e-4$ &$4.1199\e+0$ &$1\e-13$& 55 & 57 & 58.2 \\
		ALM & $3.6771\e-06$ &$3.8351\e-3$ & $7.5915\e-3$& $9\e-12$& 20& 2435&106.2\\
		\midrule 
		\multicolumn{8}{c}{$l=1000, n = 500$}\\
		RGD&$1.1373\e+02$ &$8.5525\e-4$ &$2.2721\e+0$ &$1\e-13$& 66 & 67 & 73.2 \\
		ALM & $5.0538\e-12$ &$9.8152\e-8$ & $7.7892\e-7$& $1\e-06$& 5& 384&19.2\\
		\midrule
		\multicolumn{8}{c}{$l=1000, n = 1000$}\\
		RGD&$1.4532\e-9$ &$7.6241\e-5$ &$1.1278\e-6$ &$1\e-13$& 18 & 20 & 157.4 \\
		ALM & $1.4499\e-06$ &$2.4082\e-3$ & $3.1459\e-3$& $1\e-12$& 20& 2523&717.4\\
		\midrule
		\multicolumn{8}{c}{$l=1500, n = 1000$}\\
		RGD&$4.0584\e-11$ &$1.2741\e-5$ &$1.1278\e-6$ &$1\e-13$& 18 & 19 & 158.8 \\
		ALM & $2.2747\e-10$ &$1.9010\e-6$ & $5.6926\e-6$& $9\e-10$& 20& 2523&189.3\\
\bottomrule
\end{tabular}
\end{table}

As the second application, we consider a matrix equation
\begin{equation}\label{eq:mateq}
GX = B, G\in \SPD(n),
\end{equation}
on the indefinite Stiefel manifold $\iSt_{A,I_k}$ with nonsingular $A\in \symset(n)$. If there is no information about the existence of a solution, a natural idea is to minimize the associated cost function of the form \eqref{eq:fitting}. One can see that if $B \in \iSt_{G^{-1}AG^{-1},I_k}$ then \eqref{eq:mateq} admits a unique solution $G^{-1}B$ which is also the unique global solution of the associated optimization problem. 
When $G = I_n$, it is the nearest matrix problem considered in \cite{GSAS21}; some minimal distance and averaging problems can also be converted into this form, see, e.g., \cite{FioriT2009,Fior11}. 

In the first scenario, most of the numerical data are fixed as follows. Let $(n,p,m) = (4000,3000,1000)$ and $k=10$. First, let 
MATLAB generate an $n\times n$ random matrix with default setting and then this matrix is orthogonalized to yield $V$ whose columns are denoted by $v_j,j=1,\ldots,n$. Next, we set $A = V\diag(1,\ldots,p,-m,\ldots,-1)V^T$. The right hand side is set as 
$B = G[v_1,\ldots,v_k]\diag(1/\sqrt{1},\ldots,1/\sqrt{k})$.  A set of $k$ orthogonal eigenvectors corresponding to $k$ positive eigenvalues of $A$ 
divided by the square root of the associated eigenvalue is 
used as an initial guess $X_0.$  Table~\ref{tab:matrixeq_4000} reports the result of running different matrices $G$ with the initial guess to be the scaled orthogonal eigenvectors of $A$ associated with $k$ eigenvalues of the largest magnitudes obtained by \texttt{eigs} in MATLAB. The same quantities as in the previous examples are displayed. 

We try various sets of parameter values and initial guesses with the ALM method, but the obtained result is unsatisfactory and therefore not reported here.

\begin{table}[h!]
	\centering
	\caption{Matrix equation $GX = B$ tests with $n=4000$, $p = 3000$, $m = 1000$, $k = 10$, $\Mx = G^TG$, $A =\emph{\texttt{diag}}(1,\ldots,p,-m,\ldots,-1)$,  $\emph{\texttt{rstop}} = 1\e-9$.\label{tab:matrixeq_4000}}
	\footnotesize
	\begin{tabular}{c c c c c r r r}
		\toprule
		Retr.&obj.&grad.&diff.&feas.&$\#$iter.& $\#$eval.& CPU \\ \midrule 
		\multicolumn{8}{c}{$G=\texttt{gallery}(\texttt{lehmer},n)$}\\
		\eqref{eq:Cayley_curve}&$1.333\e-12$ &$2.309\e-6$ &$7.660\e-8$ &$1\e-13$& 15 & 16 & 129.4 \\
		\eqref{eq:Rcay2}&  $1.086\e-12$ & $2.084\e-6$& $8.397\e-8$& $1\e-14$& 15 &16 & 30.5 \\
		\eqref{eq:Cayley_econ}&$2.172\e-12$ &$2.948\e-6$& $8.410\e-8$& $1\e-14$ & 15 & 16 & 30.4 \\
		\midrule
			\multicolumn{8}{c}{$G=\texttt{gallery}(\texttt{kms},n,0.5)$}\\
		\eqref{eq:Cayley_curve}&$1.596\e-22$ &$2.527\e-11$ &$2.164\e-11$ &$1\e-13$& 13 & 14 & 148.0 \\
		\eqref{eq:Rcay2}&  $1.596\e-22$ & $2.527\e-11$& $2.164\e-11$& $1\e-14$& 13 &14 & 67.4 \\
		\eqref{eq:Cayley_econ}&$1.596\e-22$ &$2.527\e-11$& $2.164\e-11$& $1\e-14$ & 13 & 14 & 61.7 \\
		\midrule
		\multicolumn{8}{c}{$G=\texttt{gallery}(\texttt{minij},n)$}\\
		\eqref{eq:Cayley_curve}&$4.717\e-7$ &$1.374\e-3$ &$5.409\e-9$ &$1\e-13$& 111 & 149 & 967.1 \\
		\eqref{eq:Rcay2}&  $7.224\e-9$ & $1.700\e-4$& $3.069\e-9$& $1\e-13$& 123 &163 & 236.4 \\
		\eqref{eq:Cayley_econ}&$9.846\e-8$ & $6.275\e-4$& $4.456\e-9$& $1\e-14$ & 113 & 154 & 216.3 \\
		\bottomrule
	\end{tabular}
\end{table}

In order to practically verify the global convergence of Algorithm~\ref{alg:non-monotone gradient}, we reuse the above model setting with random data. Namely, $V$ is 
randomly generated without any prescribed generator; $X_0$ is the set of $k$ randomly picked columns among the first $p$ columns of $V$ using \texttt{randsample}$(p,k)$ in MATLAB. Then, for each model matrix, we run ten tests. The average of quantities obtained from these tests are given in Table~\ref{tab:matrixeq_4000_random}. Apparently, the 
iterate 
generated by the proposed algorithm always converges to the global minimizer of the cost function which is also the solution to the matrix equation. 
\begin{table}[h!]
	\centering
	\caption{The average values of ten tests with random data and the  retraction 
		formula \eqref{eq:Cayley_econ} for matrix equation $GX = B$, 
		and $n=4000$, $p = 3000$, $m = 1000$, $k = 10$, $\Mx = G^TG$, $A =\emph{\texttt{diag}}(1,\ldots,p,-m,\ldots,-1)$,  $\emph{\texttt{rstop}} = 1\e-9$.\label{tab:matrixeq_4000_random}}
	\footnotesize
	\begin{tabular}{l c c c c r r r}
		\toprule
		Matr $G$ &obj.&grad.&err.&feas.&$\#$iter.& $\#$eval.& CPU \\ \midrule 
		\texttt{lehmer}&$2.856\e-14$ &$2.687\e-7$ &$4.627\e-09$ &$1\e-13$& 15.9 & 16.9 & 38.0 \\
		\texttt{kms}&$1.090\e-18$ &$1.249\e-9$ &$9.963\e-10$ &$4\e-14$& 13.3 & 14.3 & 68.5\\
		 \texttt{minij}&$3.146\e-07$ &$9.162\e-4$ &$5.676\e-09$ &$2\e-13$& 113.4 & 153.6 & 224.2\\
		\bottomrule
	\end{tabular}
\end{table}

From the examples presented in both two subsections~\ref{ssec:Example_Tracemin} and \ref{ssec:Example_MatEq}, it can be concluded that (i) the algorithm with all of its components proposed for minimization on the indefinite Stiefel manifold works as expected: it converges to a solution regardless of the initial guess. In most cases, it apparently delivers better results than that obtained by the classical augmented Lagrangian method for the equality constrained problem. Moreover, (ii) an appropriately chosen metric, enabled by the consideration of a tractable metric, always results in a faster solution: fewer iterations and CPU time. And finally, (iii) the three formulae for the Cayley retraction all work properly: the conventional one \eqref{eq:Cayley_curve} is suitable for large $k$, e.g., the indefinite orthogonal group, while the two others \eqref{eq:Rcay2} and \eqref{eq:Cayley_econ} are advisable for small $k$ as they help 
replace inversion of an $n\times n$ matrix by that of a $k$-sized matrix. 
Using these two schemes 
can sometimes result in a slight loss of feasibility. This is most probably due to the occasional instability of Sherman-Morrison-Woodbury formula as reported in, e.g., \cite{FineS2001,HaoS2021}; it can be restored using a suitable Gram-Schmidt algorithm, upon the individual case.

\section{Conclusion}
We addressed 
the minimization problem subject to a quadratic constraint of the form $X^TAX = J$ with symmetric, nonsingular matrix $A$ and  the matrix $J$ satisfying $J^2 = I$, 
in the Riemannian framework. 
Our result provided a unified framework for Riemannian optimization on the orthogonal and generalized Stiefel manifolds, $J$-orthogonal group, and more than that. 

The numerical examples showed that, on the one hand, the proposed Riemannian gradient descent scheme is a highly capable method, especially for the case that an appropriate metric is used, but on the other hand, it may suffer from slow convergence. More robust algorithms such as conjugate gradient or an efficient way to exploit the second-order information of the cost function beyond the tractable metric might be considered in future work.

\section*{Acknowledgement}
The authors would like to thank Dr. Marija Miloloža Pandur for providing the data used in the last example of subsection~\ref{ssec:Example_Tracemin}.  
Dinh Van Tiep's research was supported by Thai Nguyen University of Technology.
\bibliographystyle{siamplain}
\bibliography{references_v2}
\end{document}